\documentclass[10pt,twocolumn,twoside]{IEEEtran}
\IEEEoverridecommandlockouts 
\usepackage{graphicx}    
\usepackage{cite }  
\usepackage{color}
\usepackage{eufrak}      
\usepackage{dsfont}       
\newcommand{\halmos}	{${ }^{ }$\hfill\rule{2mm}{2mm}}

\newcommand{\EQ}{\begin{eqnarray}}
\newcommand{\EN}{\end{eqnarray}}
\newcommand{\EQQ}{\begin{eqnarray*}}
\newcommand{\ENN}{\end{eqnarray*}}
\newcommand{\nnum}{\nonumber}

\title{On the Hybrid Minimum Principle on Lie Groups and
the Exponential Gradient HMP Algorithm \thanks{ This work was supported by an NSERC Canada Discovery grant and by AFOSR.}}

\author{Farzin Taringoo and Peter E. Caines\thanks{Department of Electrical and Electronic Engineering, The University of Melbourne, ftaringoo@unimelb.edu.au.\newline Department of Electrical and Computer Engineering and the Centre for Intelligent
Machines (CIM), McGill University, Montreal, Canada, peterc@cim.mcgill.ca}}
\begin{document}
\maketitle

\begin{abstract}                
This  paper provides a geometrical derivation of the Hybrid Minimum Principle (HMP) for  autonomous  hybrid systems whose  state manifolds constitute Lie groups $(G,\star)$ which are left invariant under the controlled dynamics of the system, and whose switching manifolds  are defined as smooth embedded time invariant submanifolds of $G$. The analysis  is expressed in terms of extremal (i.e. optimal)  trajectories on the cotangent bundle of the state manifold $G$. Based upon the theory in the paper,  the Hybrid Maximum Principle (HMP) algorithm introduced in \cite{Shaikh} is extended to the so-called Exponential Gradient algorithm for systems on Lie groups. The convergence analysis for the algorithm is based upon the LaSalle Invariance Principle and simulation results illustrate their efficacy.     \end{abstract}



\section{Introduction}
Lie groups have long been considered as configuration manifolds for dynamical systems (see e.g. \cite{Agra, Bro, Abra, Bloch}), and correspondingly various control problems have been formulated for controlled systems defined on Lie groups (see e.g. \cite{Agra,Leonard1}),  including in particular optimal control problems, see \cite{Bro1, Jurd1, Bloch1, Jurd}.
 
In an independent line of research, the problem of hybrid systems optimal control (HSOC) has been studied
and analyzed in many papers, see e.g. \cite{Bran1,Riedinger,Shaikh,Sussmann,Sus1,Hab,Tomlin,Dmitruk, Xu}. In particular, \cite{Shaikh, Sussmann,  Azhmyakov} present an extension of the Hybrid Maximum Principle (henceforth referred to as Minimum due to the nature of the performance function and abbreviated as HMP) for hybrid systems
 and \cite{Shaikh} presents an iterative algorithm for trajectory optimization which is based upon the HMP necessary conditions for optimality. The HMP algorithm presented in \cite{Shaikh} is a general search method applicable to both autonomous and controlled hybrid systems, that is to say hybrid systems with discrete state switching and continuous state jumps at switching manifolds, and controlled state switching and state jumps respectively. A geometric version of Pontryagin's Minimum Principle for a general class of state  manifolds is given in \cite{Agra, Sussmann, Barbero}.   

In this paper, we  generalize   the  analysis  in \cite{Shaikh} and \cite{Taringoo5} to obtain the HMP result for left invariant hybrid systems defined on Lie groups. Our method is based upon a construction of the adjoint processes by transferring the optimal state variation at the optimal final state to the identity element of a Lie group without an a prior assumption on the existence of the adjoint processes as per \cite{Dmitruk, Sussmann,Sus1}.  In this connection we note that our analysis generalizes the method presented in \cite{Dmitruk} for hybrid control systems with open control value sets. We employ the notion of Riemannian metrics to analyze the optimal state variation through a switching manifold to obtain the discontinuity of the adjoint variable at the optimal switching state and time for autonomous hybrid systems.   The analysis in this paper can also be applied to right invariant hybrid systems where the corresponding adjoint variable will be different. This proof can be generalized to a class of autonomous hybrid systems associated with  time varying switching manifolds.

 In the last part of the paper, the HMP algorithm in \cite{Shaikh} is generalized to the so-called exponential gradient HMP algorithm by employing the notion of exponential curves on Lie groups. The convergence analysis for the proposed algorithm is based on the LaSalle Invariance Principle.
\section{Hybrid systems}
 In the following definition the standard hybrid systems framework (see e.g. \cite{Bran1,Shaikh}) is generalized to  the case where the continuous state space is a smooth manifold, where henceforth in this paper smooth means $C^{\infty}$.
\newtheorem{definition}{Definition}
\newtheorem{lemma}{Lemma}
\newtheorem{proposition}{Proposition}
\newtheorem{theorem}{Theorem}
\begin{definition}
\label{d1}
A hybrid system with autonomous discrete transitions is a five-tuple 
\EQ \bf{H}:= \{\textsl{H}=\textsl{Q}\times \mathcal{M}, \textsl{U},\textsl{F},\mathcal{S},\mathcal{J}\}\EN
where:\\
$ Q=\{1,2,3,...,|Q|\}$
is a finite set of \textit{discrete (valued) states (components)} and $\mathcal{M}$ is a smooth $n$ dimensional  Riemannian continuous (valued) state (component) manifold with associated metric $g_{\mathcal{M}}$.
\\$H$  is called the \textit{hybrid state space} of $\bf{H}$.\\
$U\subset \mathds{R}^{u}$ is a set of \textit{admissible input control values}, where  $U$ is a compact set in $\mathds{R}^u$. The set of \textit{admissible input control functions} is $\mathcal{I}:=(L_{\infty}[t_{0},t_{f}),U)$, the set of all bounded measurable functions on some interval $[t_{0},t_{f}), t_{f}< \infty$, taking values in $U$.\\
$F$ is an \textit{indexed collection of smooth, i.e. $C^{\infty}$, vector fields} $\{f_{q_{i}}\}_{q_{i}\in Q}$, where $f_{q_{i}}:\mathcal{M}\times U\rightarrow T\mathcal{M}$ is a controlled vector field assigned to each discrete state; hence each $f_{q_{i}}$ is continuous on $\mathcal{M}\times U$ and continuously differentiable on $\mathcal{M}$ for all $u\in U$.\\
$\mathcal{S}:=\{n^{k}_{\gamma}: \gamma\in Q\times Q, 1\le k \le K<\infty,  n^{k}_{\gamma}\subset \mathcal{M}\}$ is a collection of  embedded time independent pairwise disjoint switching manifolds $($except in the case 
 where $\gamma = (p,q)$ is identified with $\gamma^{'} = (q,p)$$)$ such that for any ordered pair $\gamma=(p,q), n^{k}_{\gamma}$ is an open smooth, oriented codimension 1 submanifold of $\mathcal{M}$, possibly with boundary $\partial {n}^{k}_{\gamma}$. By abuse of notation, we describe the manifolds  locally by $n^{k}_{\gamma}=\{x: n^{k}_{\gamma}(x)=0, x\in \mathds{R}^{n}\}$.\\ 
$\mathcal{J}$  shall denote the family of the \textit{state jump functions} on the manifold $\mathcal{M}$.\halmos
\end{definition}
We assume:\\
\textbf{\textit{A1}}: The \textit{initial state} $h_{0}:=(x(t_{0}),q_{0})\in H$ is such that $x_{0}=x(t_{0})\notin \mathcal{S}$ for all $q_{i}\in Q$. 

A \textit{(hybrid) input function }$u$ is defined on a half open interval
$[t_0, t_{f}), t_{f} \leq \infty$, where further $u\in \mathcal{I}$.
 A \textit{(hybrid) state trajectory} with initial state $h_{0}$ and (hybrid) input function $u$ is a triple $(\tau, q, x)$
consisting of a strictly increasing sequence of (boundary and switching) times
$\tau = (t_0, t_1, t_2,\dots)$, an associated sequence of discrete states $q = (q_0,q_1, q_2, \dots)$, and a sequence $x(\cdot) = (x_{q_0}(\cdot), x_{q_1}(\cdot), x_{q_2}(\cdot), \dots)$ of absolutely continuous
functions $x_{q_i}: [t_i, t_{i+1}) \rightarrow \mathcal{M}$ satisfying the  continuous and discrete dynamics given by the following definition.
\begin{definition}(\textit{Hybrid System Dynamics})
\label{d0}
The continuous dynamics of a hybrid system $\bf{H}$ with initial condition $h_{0}=(x_{0}, q_{0})$, input control function $u\in \mathcal{I}$ and hybrid state trajectory $(\tau, q, x)$ are specified piecewise in time via the mappings
\EQ\label{state}&&(x_{q_{i}},u):[t_i, t_{i+1})\rightarrow \mathcal{M}\times U,\nnum\\&& i=0,...,L,\quad 0<L<\infty, \EN
where $x_{q_{i}}(.)$ is an integral curve of  $ \label{state3}f_{q_{i}}(.,u(.)): \mathcal{M}\times[t_i, t_{i+1})\rightarrow T\mathcal{M}$ satisfying 
\EQ\label{state2} \dot{x}_{q_{i}}(t) = f_{q_{i}}(x_{q_{i}}(t), u(t)),
\quad a.e.  \: t \in [t_i, t_{i+1}),\nnum\EN
where $x_{q_{i+1}}(t_{i+1})$ is given recursively by 
\EQ \label{thm1}&&x_{q_{i+1}}(t_{i+1})=\lim_{t\uparrow t^{-}_{i+1}}\zeta_{q_{i},q_{i+1}}(x_{q_{i}}(t)), \nnum\\&& h_{0}=(q_{0},x_{0}),  t<t_{f}.\EN

The discrete autonomous switching dynamics are defined as follows:\\
For all $p,q$,  whenever an admissible hybrid system trajectory governed by the controlled vector field $f_{p}$ meets any given  switching manifold $n_{p,q}$ transversally, i.e. $f_{p}(x(t^{-}_{s}),t^{-}_{s}) \notin T_{x(t^{-}_{s})}\mathcal{S}$, there is an autonomous switching to the controlled vector field $f_{q}$ where this is also transversal
to  $T_{x(t^{-}_{s})}S$, corresponding to a discrete state transition $p\rightarrow q,\hspace{.2cm} p,q\in Q$.   Conversely, any autonomous discrete state transition corresponds to a transversal intersection. 

A system trajectory is not continued after a non-transversal intersection with
a switching manifold.
For an autonomous switching event from $p\in Q$ to $q\in Q$, the corresponding jump function is given by a smooth map $\zeta_{p,q}:\mathcal{M}\rightarrow \mathcal{M}$: if $x(t^{-})\in \mathcal{S}$ the state trajectory jumps to $x(t)=\zeta_{p,q}(x(t^{-}))\in \mathcal{M}$, $\zeta_{p,q}\in \mathcal{J}$, where $\zeta_{p,q}(x(t^-))$ does not lie on a switching manifold  $m_{q,r}$ for any $r \in Q$. The non-jump special case is given by $x(t)=x(t^{-})$.\halmos 

Given the definitions and assumptions above, standard arguments  give the existence and uniqueness of a hybrid state trajectory $(\tau, q, x)$, with initial state $h_{0}\in H$ and input function $u\in \mathcal{I}$, up to $T,$ defined to be the least of an explosion time or an instant of non-transversal intersection with a switching manifold, see e.g. \cite{peter}
\end{definition}

 We use the term \textit{impulsive hybrid systems} for those hybrid systems where the continuous part of the  state trajectory may have discontinuous transitions (i.e. jump) at controlled or autonomous discrete state switching times or an instant of non-transversal intersection with a switching manifold or a
zeno time, i.e.
an accumulation point of times of controlled continuous state jumps or of
controlled or autonomous discrete state switchings.

We adopt:\\
 
\textbf{\textit{A2}}:~~(Controllability) For any $ q\in Q$, all pairs of states $(x_{1}, x_{2})$ are mutually accessible in any given time period $[\hat{t},\acute{t}]\subseteq[t_{0},t_{f}]$,  via the controlled vector field $\dot{x}_{q}(t) = f_{q}(x_{q}(t), u(t)), \hspace{.1cm}$ for some $u\in \mathcal{I}=(L_{\infty}[t_{0},t_{f}],U)$.\\ 

\textbf{\textit{A3}}:~~ $\{l_{q_{i}}\}_{q_{i}\in Q}$, is a family of \textit{loss functions} such that $l_{q_{i}}\in C^{k}(\mathcal{M}\times \textit{U};\mathds{R}_{+}),k\geq1$, and $h$ is a \textit{terminal cost function} such that  $h\in C^{k}(\mathcal{M};\mathds{R}_{+}),k\geq1$.\\ 

Henceforth, Hypotheses \textbf{\textit{A1}}-\textbf{\textit{A3}} will be in force unless otherwise stated. 
Let $L$ be the number of switchings and $ u\in \mathcal{I}$ then we define the \textit{hybrid cost function} as
\EQ \label{cost}&&J(t_{0},t_{f},h_{0};L,u):=\sum^{L}_{i=0}\int^{t_{i+1}}_{t_{i}}l_{q_{i}}(x_{q_{i}}(s),u(s))ds+\nnum\\&&h(x_{q_{L}}(t_{f})),\hspace{.4cm} t_{L+1}=t_{f}<T, u\in \mathcal{I},\EN
where we observe the conditions above yield $J(t_{0},t_{f},h_{0};L,u)<\infty$.
\begin{definition}
\label{d2}
 For a hybrid system $\bf{H}$, given the data $(t_{0},t_{f},h_{0};L)$, the \textit{Bolza Hybrid Optimal Control Problem} (BHOCP) is defined as the infimization  of the hybrid cost function $J(t_{0},t_{f},h_{0};L, u)$ over  the hybrid input functions $u\in \mathcal{I}$, i.e. 
 \EQ\label{cost3} J^{o}(t_{0},t_{f},h_{0};L)=inf_{u\in\mathcal{I}}J(t_{0},t_{f},h_{0};L,u).\EN \halmos
 \end{definition}
 \begin{definition}
 \label{dm}
 A \textit{Mayer Hybrid  Optimal Control Problem} (MHOCP) is defined as the special case of the BHOCP  where the cost function given in (\ref{cost}) is evaluated only on the terminal state of the system, i.e. $l_{q_{i}}=0,\hspace{.2cm}i=1,...,L$.

 \end{definition}\halmos
 
In general, different control inputs result in different sequences of discrete states of different cardinality. However,  in this paper,  we shall restrict the infimization to be over the  class of control functions, generically denoted $\mathcal{U}\subset\mathcal{I}$,  which generates an  a priori given sequence of discrete  transition events.\\\\
We adopt the following standard notation and terminology, see e.g. \cite{Lewis}.
The time dependent flow associated to a differentiable time independent vector field $f_{q_{i}}$ is a map $\Phi_{f^{u}_{q_{i}}}$ satisfying ($f^{u}_{q_{i}}(.)$ is used here for brevity instead of $f_{q_{i}}(.,u(t))$ since the calculations are performed with respect to a given control $u$):
\EQ && \Phi_{f^{u}_{q_{i}}}:[t_i, t_{i+1})\times [t_i, t_{i+1})\times \mathcal{M}\rightarrow \mathcal{M}, \nnum\\&& (t,s,x)\rightarrow \Phi^{(t,s)}_{f^{u}_{q_{i}}}(x):= \Phi_{f^{u}_{q_{i}}}((t,s),x)\in \mathcal{M},\EN
where
 \EQ\label{1122}\Phi^{(t,s)}_{f^{u}_{q_{i}}}:\mathcal{M}\rightarrow \mathcal{M},\quad \Phi^{(s,s)}_{f^{u}_{q_{i}}}(x)=x,\EN
\EQ  \frac{d}{dt}\Phi^{(t,s)}_{f^{u}_{q_{i}}}(x)|_{t}=f_{q_{i}}\big(\Phi^{(t,s)}_{f^{u}_{q_{i}}}(x(s))\big), \hspace{.2cm}t,s\in[t_{i},t_{i+1}).\EN

We associate $T \Phi_{ f^{u}_{q_{i}}}^{(t,s)}(.)$ to  $\Phi^{(t,s)}_{f^{u}_{q_{i}}}:\mathcal{M}\rightarrow \mathcal{M}$ via the push-forward of $\Phi^{(t,s)}_{f^{u}_{q_{i}}}$. 
\EQ   T\Phi_{ f^{u}_{q_{i}}}^{(t,s)}:T_{x}\mathcal{M}\rightarrow T_{\Phi_{ f^{u}_{q_{i}}}^{(t,s)}(x)}\mathcal{M}.\EN
Following \cite{Lewis},  the corresponding \textit{tangent lift} of $f^{u}_{q_{i}}(.)$ is the time dependent vector field $f^{T,u}_{q_{i}}(.)\in TT\mathcal{M}$ on $T\mathcal{M}$

\EQ \label{10}f^{T,u}_{q_{i}}(v_{x}):=\frac{d}{dt}|_{t=s} T \Phi_{ f^{u}_{q_{i}}}^{(t,s)}(v_{x}),\quad v_{x}\in T_{x}\mathcal{M}, \EN
which is given locally as 
\EQ \label{22} f^{T,u}_{q_{i}}(x,v_{x})= \left [f^{u,i}_{q_{i}}(x)\frac{\partial}{\partial x^{i}}+(\frac{\partial f^{u,i}_{q_{i}}}{\partial x^{j}}v^{j})\frac{\partial}{\partial v^{i}}\right ]^{n}_{i,j=1},\EN
and $T \Phi_{ f^{u}_{q_{i}}}^{(t,s)}(.)$ is evaluated on $v_{x}\in T_{x}\mathcal{M}$, see \cite{Lewis}. 
The following lemma gives the relation between the push-forward of $\Phi^{(t,s)}_{f_{q_{i}}}$ and the tangent lift introduced in (\ref{22}).
 For simplicity and uniformity of notation, we use $f_{q_{i}}$ instead of $f^{u}_{q_{i}}$.

\begin{lemma}(\cite{LeeJ})
\label{ll1}
Consider $f_{q}(x,u)$ as a time dependent vector field on $\mathcal{M}$ and $\Phi^{(t,s)}_{f_{q}}$ as the corresponding flow. The flow of $f^{T,u}_{q}$, denoted by $\Psi:I\times I\times T\mathcal{M}\rightarrow T\mathcal{M},\quad I=[t_{0},t_{f}]$, satisfies
 \EQ &&\Psi(t,s,(x,v))=(\Phi^{(t,s)}_{f_{q}}(x),T\Phi^{(t,s)}_{f_{q}}(v))\in T\mathcal{M},\nnum\\&& (x,v)\in T\mathcal{M}.\EN\halmos \end{lemma}
 For a general Riemannian manifold $\mathcal{M}$, the role of the adjoint process $\lambda$ is played by a trajectory in the cotangent bundle of $\mathcal{M}$, i.e.  $\lambda(t)\in T^{*}_{x(t)}\mathcal{M}$. In analogy with the definition of the tangent lift we define the \textit{cotangent lift} which corresponds to the variation of a differential form $\alpha\in T^{*}\mathcal{M}$ along $x(t)$, see \cite{Tyner}:
\EQ \label{6} f^{T^{*},u}_{q}(\alpha_{x}):=\frac{d}{dt}|_{t=s}T^{*}_{x}\Phi^{(t,s)^{-1}}_{f^{u}_{q}}(\alpha_{x}),\quad \alpha_{x}\in T^{*}_{x}\mathcal{M},\EN
where here $x=x(t)=\Phi^{(t,s)}_{f^{u}_{q}}(x(s))$ .

 Similar to (\ref{22}), in the local coordinates $(x,p)$ of $T^{*}\mathcal{M}$, we have
\EQ \label{e} &&\hspace{-.8cm}f^{T^{*},u}_{q}(x,p)=\left [f^{u,i}_{q}(x)\frac{\partial}{\partial x^{i}}-(\frac{\partial f^{u,j}_{q}}{\partial x^{i}}p^{j})\frac{\partial}{\partial p^{i}} \right ]^{n}_{i,j=1}\hspace{-.3cm}.\EN
The mapping $T^{*}_{x}$ is the pull back defined on the differential forms on the cotangent bundle of $\mathcal{M}$. The covector $\alpha_{x}$ is an element of $T^{*}_{x}\mathcal{M}$, see \cite{Tyner}.   
The following lemma gives the connection between the cotangent lift defined in (\ref{6}) and its corresponding flow on $T^{*}\mathcal{M}$.
\begin{lemma}(\cite{LeeJ})
\label{l3}
Consider $f_{q}(x(t),u(t))$ as a time dependent vector field on $\mathcal{M}$, then the flow  
$ \Gamma: I\times I\times T^{*}\mathcal{M}\rightarrow T^{*}\mathcal{M},$ satisfies ($I=[t_{0},t_{f}]$)
\EQ\label{7} &&\Gamma(t,s,(x,p))=(\Phi^{(t,s)}_{f_{q}}(x),(T^{*}_{x}\Phi^{(t,s)^{-1}}_{f_{q}})(p)),\nnum\\&& (x,p)\in T^{*}\mathcal{M},\EN
and $\Gamma$ is the corresponding integral flow of  $f^{T^{*},u}_{q}$.\halmos
\end{lemma}
The mapping $T^{*}_{x}\Phi^{(t,s)^{-1}}_{f_{q}}$ denotes the pull back of  $\Phi^{-1}$ whose existence is guaranteed since $\Phi^{(t,s)}_{f_{q}}:\mathcal{M}\rightarrow \mathcal{M}$ is a diffeomorphism, see \cite{Barbero}.  
For a given trajectory $\lambda(t)\in T^{*}\mathcal{M}$, its variation with respect to time, $\dot{\lambda}(t)$, is  an element of $TT^{*}\mathcal{M}$. The vector field defined  in (\ref{6}) is the mapping $ f^{T^{*},u}_{q}:T^{*}\mathcal{M}\rightarrow TT^{*}\mathcal{M}$, from $\lambda(t)\in T^{*}\mathcal {M}$ to $\dot{\lambda}(t)\in TT^{*}\mathcal{M}$.
\begin{proposition}(\cite{Barbero,LeeJ})
\label{p1}
Let $f_{q}(.,u):I\times \mathcal{M}\rightarrow \mathcal{M}$ be the time dependent vector field giving rise to the associated  pair $f^{T,u}_{q},f^{T^{*},u}_{q}$; then along an integral curve of $f_{q}(.,u)$ on $\mathcal{M}$ 
\EQ \langle\Gamma,\Psi\rangle:I\rightarrow \mathds{R},\EN
is a constant map, where $\Gamma$ is an integral curve of $f^{T^{*},u}_{q}$ in $T^{*}\mathcal{M}$ and $\Psi$ is an integral curve of $f^{T,u}_{q}$ in $T\mathcal{M}$.

\end{proposition}\halmos

\vspace{.5cm}
\textit{\textbf{Elementary Control and Tangent Perturbations}}.\\
Consider the nominal control  $u(.)$ and define the perturbed control as follows:
\EQ u_{\pi(t^{1},u_{1})}(t,\epsilon):=u_{\pi}(t,\epsilon)=\left\{ \begin{array}{cc} \quad u_{1} \quad t^{1}-\epsilon \leq t\leq t^{1}\\ u(t)\quad\mbox{elsewhere} \end{array}\right.,  \EN
where $u_{1}\in U, 0\leq \epsilon$.

Associated to $u_{\pi}(.)$ we have the corresponding state trajectory $x_{\pi}(t,\epsilon)$ on $\mathcal{M}$. It may be shown that under suitable hypotheses  of the differentiability of $x_{\pi}$ with respect to $\epsilon$ at the switching times, then $\lim_{\epsilon\downarrow 0}x_{\pi}(t,\epsilon)=x(t) $ uniformly for $t_{0}\leq t \leq t_{f}$, see \cite{Piccoli} and  \cite{Lee}. However in this paper we employ the same hypotheses of the differentiability of $x_{\pi}$ before and after switching times but must accommodate the fact that there may be  a discontinuity of $\frac{d x_{\pi}(t)}{d\epsilon}|_{\epsilon=0}$ at any switching time $t_{i}$.
The flow resulting from the perturbed control is defined by:
\EQ&&\hspace{-.5cm} \Phi_{\pi, f_{q}}^{(t,s),x}(\epsilon):[0,\tau]\rightarrow \mathcal{M},\quad x\in \mathcal{M},t,s\in [t_{0},t_{f}], \tau\in \mathds{R}_{+},\nnum\\&&\hspace{-.5cm}\EN
where $\Phi_{\pi, f_{q}}^{(t, s), x}(.)$ is the flow corresponding to the perturbed control $u_{\pi}(t,\epsilon)$, i.e. $\Phi_{\pi, f_{q}}^{(t, s), x}(\epsilon):=\Phi_{f_{q}^{u_{\pi}(t,\epsilon)}}^{(t, s)}(x(s))$.
The following lemma gives the formula of the variation of  $\Phi_{\pi, f_{q}}^{(t,s),x}(.)$ at $\epsilon=0^{+}$. 
Recall that the point $t^{1}\in (t_{0},t_{f})$ is called Lebesgue point of  $u(.)$ if, (\cite{Agra}):
\EQ\lim_{s\rightarrow t^{1}}\frac{1}{|s-t^{1}|}\int^{s}_{t^{1}}|u(\tau)-u(t^{1})|d\tau=0. \EN 
For a $u\in L_{\infty}([t_{0},t_{f}],U)$, $u$ may be modified on a set of measure zero so that all points are   Lebesgue and the value function is unchanged (see \cite{Rudin}, page 158).
\begin{lemma}(\hspace{-.01cm}\cite{Barbero})
\label{l222}
For a Lebesgue time $t^{1}$, the curve $ \Phi_{\pi, f_{q}}^{(t,s),x}(\epsilon):[0,\tau]\rightarrow \mathcal{M}$ is differentiable at $\epsilon=0$ and the corresponding tangent vector $\frac{d}{d\epsilon}\Phi_{\pi, f_{q}}^{(t^{1},s),x}|_{\epsilon=0}$ is given by
\EQ &&\frac{d}{d\epsilon}\Phi_{\pi, f_{q}}^{(t^{1},s),x}|_{\epsilon=0}=\nnum\\&&f_{q}(x(t^{1}),u_{1})-f_{q}(x(t^{1}),u(t^{1}))\in T_{x(t^{1})}\mathcal{M}.\EN\halmos
\end{lemma}

The tangent vector $ f_{q}(x(t^{1}),u_{1})-f_{q}(x(t^{1}),u(t^{1}))$ is called the \textit{elementary perturbation vector} associated to the perturbed control $u_{\pi}$ at $(x(t),t)$.
The displacement of the tangent vectors at $x\in \mathcal{M}$ is by the push-forward defined on the vector field $f_{q}$. 

\begin{proposition}(\hspace{-.01cm}\cite{Barbero})
\label{p2}
Let $\Psi:[t^{1}, t_{f}]\rightarrow T\mathcal{M}$ be the integral curve of $f^{T,u}_{q}$ with the initial condition $\Psi(t^{1})=[f_{q}(x(t^{1}),u_{1})-f_{q}(x(t^{1}),u(t^{1}))]\in T_{x(t^{1})}\mathcal{M}$, then 
\EQ \frac{d}{d\epsilon}\Phi_{\pi, f_{q}}^{(t,t^{1}),x}|_{\epsilon=0}=\Psi(t),\quad t\in[t^{1},t_{f}].\EN\halmos
\end{proposition}
By the result above and Lemma \ref{ll1} we have
\EQ&&\frac{d}{d\epsilon}\Phi_{\pi, f_{q}}^{(t,t^{1}),x}|_{\epsilon=0, x(t)\in \mathcal{M}}=\nnum\\&&T\Phi^{(t,t^{1})}_{f_{q}}([f_{q}(x(t^{1}),u_{1})- f_{q}(x(t^{1}),u(t^{1}))])\in T_{x(t^{1})}\mathcal{M}.\nnum\\\EN
\section{Control Systems on Lie Groups}
In this section we introduce control systems on Lie groups and then extend the definition of hybrid systems above to that of hybrid systems defined on Lie groups.
 
\subsection{Lie Groups and Lie Algebras}
\vspace*{.5cm}
 
We recall that the Lie algebra $\mathcal{L}$ (see \cite{Lewis, Varad}) of a Lie group $G$ (see \cite{Lewis, Varad}) is the tangent space at the identity element $e$ with the associated Lie bracket defined on the tangent space of $G$, i.e.  $\mathcal{L}=T_{e}G$. 
A vector field $X$ on $G$ is called \textit{left invariant } if
\EQ \forall g_{1},g_{2}\in G,\quad X(g_{1}\star g_{2})=TL_{g_{1}}X(g_{2}),\EN
where $L_{g}:G\rightarrow G,\hspace{.2cm} L_{g}(h)=g\star h,\hspace{.2cm} TL_{g}:T_{g_{2}}G\rightarrow T_{g\star g_{2}}G$ which immediately imply $X(g\star e)=X(g)=TL_{g}X(e)$.
Corresponding to a left invariant vector field $X$, we define the \textit{exponential map} as follows:
\EQ exp:\mathcal{L}\rightarrow G,\quad \exp(tX(e)):=\Phi(t,X),t\in \mathds{R},\EN

where $\Phi(t,X)$ is the solution of $\dot{g}(t)=X(g(t))$ with the boundary condition $g(0)=e$. The following theorem gives the flow of a left invariant vector field with an arbitrary initial state $g\in G$.
\begin{theorem}(\hspace{-.01cm}\cite{Varad})
\label{t1}
Let $G$ be a Lie group with the corresponding Lie algebra $\mathcal{L}$, then for a left invariant vector field $X$
\EQ \Phi(t,X,g)=L_{g}\circ \exp(tX(e)),\EN
where $\Phi(t,X,g)$ is the flow of $X$ starting at $g\in G$.  
\end{theorem}
A left invariant control system defined on a given Lie group $G$ is defined as follows: (see \cite{Lewis,Jurd, Card})
\EQ \dot{g}(t)=f(g(t),u)=TL_{g(t)}f(e,u),\quad g(t)\in G, u\in \mathds{R}^{u},\EN
where $f(g(t),u)$ is a left invariant vector field on $G$.
Similar to left invariant systems,  right invariant systems are defined. In this paper we only consider hybrid systems where the associated vector fields are left invariant, however  the analysis can also be applied to right invariant hybrid systems. 
\subsection{Left Invariant Optimal Control Systems}
A \textit{Bolza left invariant optimal control problem} is an optimal control problem where\\ (i): the ambient state manifold $\mathcal{M}$ is a Lie group $G$,\\
(ii): the corresponding vector field $f_{q}$ is a left invariant vector field defined on $G$ such that for any given $u\in \mathcal{U}$
\EQ \label{state33}f_{q}(.,u(.)):G\times [t_{0}, t_{f}]\rightarrow TG,\EN
and\\
(iii): the cost function is defined by
\EQ \label{cost11}J:=\int^{t_{f}}_{t_{0}}l_{q}(g(s),u(s))ds,\quad u\in \mathcal{U},\EN
where $l_{q}(g(s),u(s))$ is assumed to be left invariant i.e. $l_{q}(L_{h}g(s),u(s))=l_{q}(g(s),u(s))$.
In general, a Bolza problem can be converted to a Mayer problem using an auxiliary state variable in the dynamics, see \cite{Shaikh} and \cite{Barbero}. 
 The following lemma gives the equivalence of a Bolza problem defined on a Lie group $G$ and its Mayer extension.
 Consider a left invariant  Optimal Control Problem (OCP) defined on a Lie group $G$ with the following dynamics and cost function:
 \EQ \dot{g}(t)=f(g(t),u),\quad g(t)\in G, u\in \mathds{R}^{u},\EN
 \EQ J=\int^{t_{f}}_{t_{0}}l(g(s),u(s))ds.\EN
Then the Mayer problem associated to the optimal control problem above is defined on the Lie group $G\times \mathds{R}$ and the corresponding dynamics are left invariant.
The state space equation of the Mayer problem corresponding to the Bolza problem is as follows:
 \EQ  \left(\begin{array}{cc}\dot{g}\\\dot{z}\end{array}\right)=\left(\begin{array}{cc}f(g(t),u(t))\\l(g(t),u(t))\end{array}\right)=F(\bar{g}(t),u),\EN
 where $\bar{g}=(g,z),\quad g\in G,z\in \mathds{R}$.
 The group action defined on $G\times \mathds{R}$ is given as follows:
 \EQ (g_{1},z_{1})\bar{\star}(g_{2},z_{2})=(g_{1}\star g_{2},z_{1}+z_{2}),\EN
 where $\star$ corresponds to the group action of $G$ and $\bar{\star}$ is the group action of $G\times \mathds{R}$.
 Since $G$ is a Lie group it follows that $(G\times \mathds{R},\bar{\star})$  is also a Lie group. It remains to show that $F(\bar{g},u)$ is left invariant. The left translation on $G\times \mathds{R}$ is defined by
 \EQ L_{\bar{g}}(\bar{h})=(L_{g}h,z_{g}+z_{h}),\quad  \bar{g}=(g,z_{g}),\bar{h}=(h,z_{h}),\EN
 therefore
 \EQ TL_{\bar{g}}F(\bar{h},u)&=&TL_{g}f(h,u)\oplus l(L_{g}h,u)\nnum\\&=&f(g\star h,u)\oplus l(g\star h,u)=F(\bar{g}\star \bar{h},u),\nnum\\\EN
 which shows that  $F(\bar{g},u)$ is left invariant since $f(g,u)$ and $l(g,u)$ respectively are left invariant.

\section{The Pontryagin Minimum Principle on Lie Groups}
Optimal control problems on Lie groups have been addressed in \cite{Bro,Bro1,Jurd1,Jurd}. In this section we review the Minimum Principle results presented in \cite{Jurd1} for optimal control problems defined on a Lie group $G$. As shown in \cite{Jurd}, the left translation gives an isomorphism between $TG$ and $G\times \mathcal{L}$. Since 
$L_{g^{-1}}$ maps $g$ to $e$,  $TL_{g^{-1}}:T_{g}G\rightarrow T_{e}G=\mathcal{L}$ is the corresponding isomorphism.
 This statement also holds between $T^{*}G$ and $G\times \mathcal{L}^{*}$ where $\mathcal{L}^{*}$ is the dual space of the Lie algebra $\mathcal{L}$. 
 \begin{definition}
 The isomorphism  between $\mathcal{L}^{*}\times G$ and $T^{*}G$ is  denoted  by $\mathfrak{I}^{*}$, where
\EQ \hspace{0cm}(1)\quad \mathfrak{I}^{*}(X,g)\in T^{*}G,\quad X\in \mathcal{L}^{*},g\in G,\EN
\EQ\label{II}  \hspace{0cm}(2) \quad\mathfrak{I}^{*}(.,g)&:=&\mathfrak{I}^{*}_{g}:=T^{*}L_{g^{-1}}:\mathcal{L}^{*}\rightarrow T^{*}_{g}G \hspace{.25cm}\nnum\\&&\mbox{is a linear isomorphism}.\EN
\end{definition}\halmos

We use the equivalence $T^{*}G\approx G\times \mathcal{L}^{*}$  associated to the isomorphism above to construct  Hamiltonian functions on Lie groups.
\\

\textit{\textbf{Hamiltonian Systems on $T^{*}\mathcal{M}$ and $T^{*}G$}}.\\
By definition, for an optimal control problem defined on an $n$ dimensional  differentiable manifold $\mathcal{M}$,  a Hamiltonian function is defined as a smooth function  $H:T^{*}\mathcal{M}\times U\rightarrow \mathds{R}$, see \cite{Agra, Jurd}.  
The associated Hamiltonian vector field $\overrightarrow{H}$ is defined as follows (see \cite{Agra}):
\EQ \omega_{\lambda}(.,\overrightarrow{H})=dH,\quad \lambda\in T^{*}M,\EN
where $\omega_{\lambda}$ is the symplectic form defined on $T^{*}\mathcal{M}$ which is locally written as follows:
\EQ \omega_{\lambda}=\sum^{n}_{i=1}d\zeta_{i}\wedge dx_{i},\EN
and $(\zeta,x)$ is the local coordinate representation of $\lambda$ in $T^{*}\mathcal{M}$.

The Hamiltonian system of the ODE corresponding to $H$ is
\EQ \dot{\lambda}=\overrightarrow{H}(\lambda),\EN
where locally we have
\EQ\left\{ \begin{array}{ll} \dot{x_{i}}=\frac{\partial H}{\partial \zeta_{i}},\quad i=1,...,n, \\
        \dot{\zeta_{i}}=-\frac{\partial H}{\partial x_{i}},\quad i=1,...,n.\end{array} \right.\EN
        
    Similar to the case of Hamiltonian systems on smooth manifolds we can define  Hamiltonian functions for left invariant vector fields on the cotangent bundle of a Lie group $G$. 
 A Hamiltonian function for a  left invariant vector field $X$ on $G$ is defined by
 \EQ H_{X}(g,\lambda):=\langle \lambda, X(e)\rangle=\langle \lambda, TL_{g^{-1}}X(g)\rangle,\quad \lambda\in \mathcal{L}^{*}.\EN
 The preceding identification uses  $TT^{*}G\simeq T(G\times \mathcal{L}^{*})=(G\times \mathcal{L})\times (\mathcal{L}^{*}\times \mathcal{L}^{*})$, therefore the tangent vector at  $(g,\lambda)\in G\times \mathcal{L}^{*}$ is an element of $\mathcal{L}\times\mathcal{L}^{*}$ denoted by $T=(X,\gamma)$.  The symplectic form $\omega$ along a given curve $\Gamma(t)\in T^{*}G$  satisfies the following equation, see \cite{Agra,Jurd}:
 \EQ\hspace{-1cm}&&\omega_{\Gamma}(T_{1}(\Gamma),T_{2}(\Gamma))=\langle\gamma_{2}(t),X_{1}(t)\rangle\nnum\\&&-\langle\gamma_{1}(t),X_{2}(t)\rangle-\langle\lambda(t),[X_{1}(t),X_{2}(t)]\rangle,\EN
 where $T_{i}=(X_{i},\gamma_{i})$. Similar to Hamiltonian systems on $T^{*}\mathcal{M}$, the Hamiltonian vector field $\overrightarrow{H}$ on $G\times \mathcal{L}^{*}$ satisfies the following equation
 \EQ\label{vec} dH=\omega_{\Gamma(t)}(.,\overrightarrow{H}).\EN
 \begin{definition}
 For each $\eta,\zeta\in \mathcal{L}$ we define
 \EQ ad_{\zeta}:\mathcal{L}\rightarrow \mathcal{L},\hspace{.2cm}  ad_{\zeta}\eta=[\zeta,\eta]. \EN
 For each $\zeta,\eta\in \mathcal{L}, \gamma\in\mathcal{L}^{*}$,  $ad^{*}$ is defined by
 \EQ \langle ad^{*}_{\zeta}(\gamma),\eta\rangle:=\langle \gamma,ad_{\zeta}(\eta)\rangle.\EN
 \end{definition}\halmos
 
 For more information about the definition above see \cite{Abra, Varad}.
 
 The following theorem gives the Minimum Principle for optimal control problems defined on Lie groups.
 
 \begin{theorem}(\hspace{-.01cm}\cite{Jurd})
 For a left invariant optimal control problem defined by (\ref{state33}) and (\ref{cost11}), along the optimal state and optimal control $g^{o}(t),u^{o}(t)$, there exists a nontrivial adjoint curve $\lambda^{o}(t)\in\mathcal{L}^{*}$ such that the following equations hold:
 \EQ H(g^{o}(t),\lambda^{o}(t),u^{o}(t))\leq H(g^{o}(t),\lambda^{o}(t),u),\quad \forall u\in U,\EN
 and locally
 \EQ \label{sta}\frac{dg^{o}}{dt}=TL_{g^{o}(t)}\big(\frac{\partial H}{\partial \lambda}\big),\EN
 \EQ\label{lam}  \frac{d\lambda^{o}}{dt}=-(ad)^{*}_{\frac{\partial H}{\partial \lambda}}(\lambda^{o}(t)),\EN
 where $ H(g,\lambda,u):=\langle \lambda,TL_{g^{-1}}f(g,u)\rangle$.
 \end{theorem}\halmos

 For left invariant vector fields $\frac{\partial H}{\partial g}=0$ for a Hamiltonian function defined on $G\times \mathcal{L}^{*}$; but in general, that is to say for not necessarily left invariant vector fields, the integral curve of the Hamiltonian vector field, i.e. (\ref{sta}) and (\ref{lam}), satisfies the following equations (see \cite{Jurd}):
 \EQ \frac{dg}{dt}=TL_{g(t)}\big(\frac{\partial H}{\partial \lambda}\big),\EN
 \EQ  \frac{d\lambda}{dt}=-T^{*}L_{g(t)}\big(\frac{\partial H}{\partial g}\big)-(ad)^{*}_{\frac{\partial H}{\partial \lambda}}(\lambda(t)).\EN
Since the tangent space of $T^{*}G$ is identified with $\mathcal{L}\times \mathcal{L}^{*}$,
by the definition of the Hamiltonian $H:G\times \mathcal{L}^{*}\rightarrow \mathds{R}$, it is noted that $\frac{\partial H}{\partial \lambda}\in \mathcal{L}^{*^{*}}=\mathcal{L}$ and $\frac{\partial H}{\partial g}\in T^{*}_{g}G$.

 
 \section{Hybrid Systems on Lie Groups}
 The definition of hybrid systems on Lie groups is the specialization of that of hybrid systems given in Definition \ref{d1} where the ambient manifold $\mathcal{M}$ is replaced by a Lie group $G$. Here we only consider a hybrid system consisting of two different discrete states with the associated left invariant  vector fields $f_{q_{0}},f_{q_{1}},$ as follows:
  
 \EQ\label{1} \dot{g}(t)=f_{q_{0}}(g(t),u(t)),\hspace{.1cm} \dot{g}(t)=f_{q_{1}}(g(t),u(t)),\hspace{.1cm} u(t)\in \mathcal{U}.\EN
 The switching manifold $\mathcal{S}$ associated to the autonomous discrete state change is considered to be a submanifold  of $G$ which is by definition a regular Lie subgroup.  
 The hybrid cost function is defined by
 \EQ\label{hcost} J=\sum^{1}_{i=0}\int^{t_{i+1}}_{t_{i}}l_{q_{i}}(g_{q_{i+1}}(s),u(s))ds+h(g_{q_{L}}(t_{f})),\hspace{.05cm} u\in \mathcal{U},\EN
 where $l_{i}, i=0,1$ are left invariant smooth functions on $G$.
Similar to the proof in \cite{Taringoo5}, we apply the needle control variation in two different steps.  First, the control needle variation is applied after the optimal switching time so there is no state propagation along  the state trajectory through the switching manifold.  Second, the control needle variation is applied  before the optimal switching time. In this case there exists a state variation propagation through the switching manifold, see \cite{Shaikh}, \cite{Taringoo5}. 

Recalling  assumption \textbf{\textit{A2}} in the Bolza problem and assuming the existence of optimal controls for each pair of given switching state and switching time, let us define  a function $v:G\times (t_{0},t_{f})\rightarrow \mathds{R}$  for a hybrid system with one autonomous switching, i.e. $L=1$, as follows:
\EQ \label{v}v(g,t)=inf_{u\in \mathcal{U}}J(t_{0},t_{f},h_{0}, u),\EN
where 
\EQ g=\Phi^{(t^{-},t_{0})}_{f_{q_{0}}}(g_{0})\in \mathcal{S}\subset G.\nnum\EN

\subsection{Non-Interior Optimal Switching States}
\label{s2}
In general the hybrid value function for a Mayer type problem attains its minimum on the boundary of the attainable switching states on the switching manifold and hence is not differentiable.  In this case the discontinuity of the adjoint process in the HMP statement is given in terms of a normal vector at the switching time on the switching manifold.
In order to have a normal vector $N$ on the switching manifold we need to define a Riemannian metric on $G$. A left invariant Riemannian metric $\textbf{G}$ on $(G,\star)$ satisfies the following equation.
\EQ\textbf{G}(g)(X,Y)=\textbf{G}(h\star g)(TL_{h}(X),TL_{h}(Y)), \EN
where $X,Y\in T_{g}G$. Consider a generic inner product $\textbf{I}$ on $\mathcal{L}$, where $\textbf{I}:\mathcal{L}\times \mathcal{L}\rightarrow \mathds{R}$. The following theorem gives a Riemannian metric with respect to $\textbf{I}$ defined on $\mathcal{L}$.
\begin{lemma}(\hspace{-.01cm}\cite{Lewis})
\label{l10}
An inner product $\textbf{I}$ on $\mathcal{L}$ determines a smooth left invariant Riemannian metric $\textbf{G}$ on $G$ as follows:
\EQ \textbf{G}(g)(X,Y):=\textbf{I}(TL_{g^{-1}}X,TL_{g^{-1}}Y),\EN
where $X,Y\in T_{g}G$. \halmos 
\end{lemma} 

A normal vector $N$ at the switching state $g(t_{s})$ on $\mathcal{S}$ satisfies 
\EQ \textbf{G}(g)(N,Y)=0,\hspace{.2cm}\forall Y\in T_{g(t_{s})}\mathcal{S}\subset T_{g(t_{s})}G,\EN
where by Lemma \ref{l10} we have $\textbf{I}(TL_{g^{-1}}N,TL_{g^{-1}}Y)=0$. By the linearity of the inner product $\textbf{I}$ on the vector space $\mathcal{L}$, we can defined the following one form 
\EQ D_{g}N:\mathcal{L}\rightarrow \mathds{R},\hspace{.2cm}D_{g}N:=\textbf{I}(TL_{g^{-1}}N,.)\in \mathcal{L}^{*}.\EN
The following lemma shows that the one form $\textbf{G}_{g}(N,.)$ is the pullback of $D_{g}N=\textbf{I}(TL_{g^{-1}}N,.)$ under the map $T^{*}L_{g^{-1}}$.
\begin{lemma}
\label{l11}
For a Lie group $(G,\star)$ associated with an inner product $\textbf{I}$ on $\mathcal{L}$ we have
\EQ\forall g\in G,  \hspace{.2cm}\textbf{G}(g)(N,.)=T^{*}L_{g^{-1}}D_{g}N\in T^{*}_{g}G,\EN
\end{lemma}

\begin{proof}
We show that for all $ X\in T_{g}G$, $\textbf{G}(g)(N,X)=\langle T^{*}L_{g^{-1}}D_{g}N,X\rangle$. 
Obviously $TL_{g^{-1}}:T_{g}G\rightarrow \mathcal{L}$, therefore
\EQ\langle T^{*}L_{g^{-1}}D_{g}N,X\rangle=\langle T^{*}L_{g^{-1}}\textbf{I}(TL_{g^{-1}}N,.),X\rangle ,\EN
By the definition of pullbacks, see \cite{Lee}, we have
\EQ\langle T^{*}L_{g^{-1}}\textbf{I}(TL_{g^{-1}}N,.),X\rangle&=&\langle \textbf{I}(TL_{g^{-1}}N,.),TL_{g^{-1}}X\rangle\nnum\\&=&\textbf{I}(TL_{g^{-1}}N,TL_{g^{-1}}X)\nnum\\&=&\textbf{G}(g)(N,X),\EN
where the second equality comes from the definition of $\textbf{I}$.
\end{proof}

The following theorem gives the HMP statement for hybrid systems defined on Lie groups in the case of non-differentiability in all directions of the value function. It is the main result of this section and will be established by a sequence of lemmas.\\
\begin{theorem}

\label{t22}
Consider a hybrid system satisfying the hypotheses \textbf{\textit{A1, A2, A3}} on a Lie group $G$  and an  embedded switching submanifold $\mathcal{S}\subset G$ with an associated inner product $\textbf{I}:\mathcal{L}\times \mathcal{L}\rightarrow \mathds{R}$. Then  corresponding to an optimal  control and optimal trajectory $(u^{o}(t),g^{o}(t))$ for a given MHOCP,  there exists a nontrivial $\lambda^{o}\in \mathcal{L}^{*}$ along the optimal state trajectory such that:
\EQ \hspace{0cm}H_{q_{i}}(g^{o}(t),\lambda^{o}(t),u^{o}(t))\leq H_{q_{i}}(g^{o}(t),\lambda^{o}(t),u_{1}),\nnum\\ \forall u_{1}\in U, t\in[t_{0},t_{f}],i=0,1, \EN
and at the optimal switching state and switching time $(g^{o}({t_{s}}),t_{s})$ we have 

\EQ\label{kirkir} \lambda^{o}(t^{-}_{s})=\lambda^{o}(t_{s})+\mu\textbf{I}(TL_{g^{o^{-1}}(t_{s})}N,.)\in \mathcal{L}^{*}.\EN
and the continuity of the Hamiltonian is satisfied as follows
\EQ\hspace{-.5cm} &&H_{q_{0}}(g^{o}(t^{-}_{s}),\lambda^{o}(t^{-}_{s}),u^{o}(t^{-}_{s}))=\nnum\\&&H_{q_{1}}(g^{o}(t_{s}),\lambda^{o}(t_{s}),u^{o}(t_{s})).\EN
The optimal adjoint variable $\lambda^{o}$ satisfies
 \EQ \label{sta1}\frac{dg^{o}}{dt}=TL_{g^{o}(t)}\big(\frac{\partial H_{q_{i}}}{\partial \lambda}\big),\quad  \frac{d\lambda^{o}}{dt}=-(ad)^{*}_{\frac{\partial H_{q_{i}}}{\partial \lambda}}(\lambda^{o}(t)),\nnum\\ t\in[t_{i},t_{i+1}), q_{i}\in Q,\EN
 where 
 \EQ H_{q_{i}}(g,\lambda,u):=\langle \lambda,TL_{g^{-1}}f_{q_{i}}(g,u)\rangle.\EN
\end{theorem}\halmos

It should be noted that in the case which the normal vector is not uniquely given, the discontinuity of the adjoint process is given by
 \EQ\lambda^{o}(t^{-}_{s})-\lambda^{o}(t_{s}))\in T^{*}L_{g^{o}(t_{s})}(T^{*^{\perp}}_{g^{0}(t^{-}_{s})}\mathcal{S}),\EN
where 
\EQ T^{*^{\perp}}_{g}\mathcal{S}:=\{\alpha\in T^{*}_{g}G,\hspace{.1cm} s.t. \hspace{.2cm} \forall X\in T_{g}\mathcal{S},\langle\alpha,X\rangle=0\}.\EN

In order to prove Theorem \ref{t22}, we employ the notion of control needle variation which has been widely used in the optimal control literature, see \cite{ Barbero, Lee, Agra}. 

 \subsection{Control Needle Variation}
 
 Similar to the control needle variation introduced in the proof of the Hybrid Maximum Principle in \cite{Shaikh}, we introduce the following control needle variation for a left invariant control system. 
\EQ u_{\pi}(t,\epsilon):=u_{\pi(\epsilon,u_{1})}(t,\epsilon)=\left\{ \begin{array}{ll} u_{1} \quad  t^{1}-\epsilon\leq t\leq t^{1}\\
        u^{o}(t)\quad \mbox{elsewhere}\end{array}  \right., \EN
        where $u_{1}\in U$.
         Let us denote the state flow of the left invariant control system $\dot{g}(t)=f(g,u)$ as $g(t)=g(t,s,g_{0})$ where $s$ is the initial time and $g_{0}$ is the initial state.
        
        Due to the needle variation, the perturbed control system is given by 
        \EQ \label{kk}\dot{g}_{(\pi,\epsilon)}(t)=f(g_{(\pi,\epsilon)}(t),u_{\pi}(t)),\quad t\in [t_{0},t_{f}].\EN
        Associated to $u_{\pi}(.,.)$ we have the corresponding state trajectory $g_{\pi}(.,.)$ on $G$. It may be shown under suitable hypotheses, $lim_{\epsilon\rightarrow 0}g_{\pi}(t,\epsilon)=g(t) $ uniformly for $t_{0}\leq t \leq t_{f}$, see \cite{Piccoli} and  \cite{Lee}. 
Following (\ref{kk}),  the flow resulting from the perturbed control is defined as:
\EQ &&g_{\pi, f}^{(t, s), g}(.):[0,\tau]\rightarrow G,\quad g\in G,t,s\in [t_{0},t_{f}], \tau\in \mathds{R}^{+},\nnum\\&& g_{\pi, f}^{(t, s), g}(\epsilon)\in G,\EN
where $g_{\pi, f}^{(t, s),g}(.)$ is the flow corresponding to the perturbed control $u_{\pi}(t,\epsilon)$, i.e. \\$g_{\pi, f}^{(t, s), g}(\epsilon):=g_{f^{u_{\pi}(t,\epsilon)}}^{(t, s)}(g(s))$.
        The following theorem gives the state variation of a left invariant control system with respect to a control needle variation.  
        \begin{lemma}
\label{l1}
For a Lebesgue time $t^{1}$, the curve \\$ g_{\pi, f}^{(t,s),g}(.):[0,\tau]\rightarrow G$ is differentiable at $\epsilon=0$ and the corresponding tangent vector $\frac{d}{d\epsilon}g_{\pi, f}^{(t^{1},s),g}|_{\epsilon=0}, $ is 
\EQ&& \hspace{-1cm}\frac{d}{d\epsilon}g_{\pi, f}^{(t^{1},s),g}|_{\epsilon=0}=TL_{g(t^{1})} \big(f(e,u_{1})-f(e,u^{o}(t^{1}))\big), \nnum\\&&s\in[t_{0},t^{1}).\EN
    \end{lemma}
    \begin{proof}
     This follows from the left invariance property of $f$ since, by Lemma \ref{l222}, the state variation with respect to the control needle variation is given by   
       \EQ\label{left} &&f(g(t^{1}),u_{1})-f(g(t^{1}),u^{o}(t^{1}))=\nnum\\&&TL_{g(t^{1})} \big(f(e,u_{1})-f(e,u^{o}(t^{1}))\big),\EN 
which completes the proof.
    \end{proof}
    
The following lemma gives the state variation at an arbitrary time $t$, where $t^{1}<t$, for a non-hybrid left invariant control system.
\begin{lemma}
\label{l2}
Let $g_{(\pi,\epsilon)}:[t_{0}, t_{f}]\rightarrow G$ be a solution of $\dot{g}_{(\pi,\epsilon)}(t)=f(g_{(\pi,\epsilon)}(t),u_{\pi}(t))$ then for $t^{1}<t\leq t_{f}$
\EQ&&\hspace{0cm}\frac{d}{d\epsilon}g_{\pi, f}^{(t,t^{1})}(g)|_{\epsilon=0}=TR_{\exp((t-t^{1})f(e,u^{o}))}\circ \nnum\\&&TL_{g(t^{1})} \left(f(e,u_{1})-f(e,u^{o}(t^{1}))\right)\in T_{g(t)}G,\EN
where $TR_{\exp((t-t^{1})f(e,u^{o}))}$ is the push forward of the right translation $R_{\exp((t-t^{1})f(e,u^{o}))}$ at $g(t^{1})$ and $g(t_{0})=g$.
\end{lemma}
\begin{proof}
As shown in \cite{Barbero} for a given control system on a differentiable manifold $\mathcal{M}$, the state variation at time $t$ where $t^{1}<t$ is given as follows:
\EQ&&\hspace{-0cm}\frac{d}{d\epsilon}\Phi_{\pi, f}^{(t,t^{1})}(x)|_{\epsilon=0}=\nnum\\&&T\Phi^{(t,t^{1})}_{f}([f(x(t^{1}),u_{1})-f(x(t^{1}),u^{o}(t^{1}))])\in T_{x(t)}\mathcal{M},\nnum\\\EN
The push-forward of $\Phi_{ f}^{(t,t^{1})}(\cdot)$, i.e. $T\Phi^{(t,t^{1})}_{f}$, is computed along the nominal state trajectory with respect to the control $u^{o}(.)$ and is evaluated at $x(t^{1})$. For a left invariant control system evolving on $G$, based on Definition \ref{d5} and Theorem \ref{t1} we have
\EQ \hspace{-0cm}g_{(\pi,\epsilon)}(t)&=&g_{0}\circ \exp(tf(e,u_{\pi}))\nnum\\&=&g_{0}\circ \exp\big((t^{1})f(e,u_{\pi})+(t-t^{1})f(e,u_{\pi})\big).\nnum\\\EN 
Since $u_{\pi}(t)=u^{o}(t),\quad t\in[t^{1},t_{f}]$, by the one parameter subgroup property of $exp$ (see \cite{Varad}) we have 
\EQ g_{(\pi,\epsilon)}(t)=g_{(\pi,\epsilon)}(t^{1})\circ \exp((t-t^{1})f(e,u^{o})),t^{1}<t\leq t_{f}.\EN
Therefore,  by evaluating the push forward of composition maps,  we have
\EQ &&\frac{d}{d\epsilon}g_{\pi, f}^{(t,t^{1})}(g)|_{\epsilon=0}=TR_{\exp((t-t^{1})f(e,u^{o}))} \times\nnum\\&&\big(\frac{d}{d\epsilon}g_{\pi, f}^{(t^{1},s)}(g)|_{\epsilon=0}\big),\hspace{.2cm} s\in[t_{0},t^{1}),\EN
which together with  Lemma \ref{l1} and (\ref{left}) yields the statement. 
\end{proof}

We analyze the HOCP  with the cost defined in (\ref{1}) and (\ref{hcost}) by defining a differential form of the penalty function $h(.)$ which is differentiable by the hypotheses.
Let us denote \EQ dh:= \frac{\partial h}{\partial g}\in \Omega(G),\EN
where $\Omega(G)$ is the set of smooth one forms on $G$.
In order to use the methods  introduced in \cite{ Barbero, Lee, Agra}, we prove the following lemma using the optimal control $u^{o}(.)$ and the associated final state $g^{o}(t_{f})$.
We denote $t_{s}(\epsilon)$ as the associated switching time corresponding to $u_{\pi}(t,\epsilon)$ which is assumed to be differentiable with respect to $\epsilon$ for all $u\in U$.

\begin{lemma}
\label{l6}
For a Hybrid Optimal Control Problem (HOCP) defined on a Lie group $G$, at the optimal final state of the trajectory $g^{o}(t)$ we have
\EQ &&\langle \mathfrak{I}^{*^{-1}}_{g^{o}(t_{f})}\big(dh(g^{o}(t_{f}))\big),TL_{g^{o^{-1}}(t_{f})}\big(v_{\pi}(t_{f})\big)\rangle\geq 0,\hspace{.2cm}\nnum\\&&\forall v_{\pi}(t_{f})\in K_{t_{f}},\EN

where
\EQ\label{cone1} &&\hspace{-.2cm}K^{1}_{t_{f}}=\bigcup_{t_{s}\leq t<t_{f}}\bigcup_{u_{1}\in U}TR_{\exp((t_{f}-t)f_{q_{1}}(e,u^{o}))}\circ\nnum\\ &&\hspace{-.2cm} TL_{g^{o}(t)}\big(f_{q_{1}}(e,u_{1})-f_{q_{1}}(e,u^{o}(t))\big)\subset T_{g^{o}(t_{f})}G,\hspace{.1cm} t\in[t_{s},t_{f}],\nnum\\\EN
and 
\EQ\label{cone11}&&\hspace{-.6cm} K^{2}_{t_{f}}=\hspace{-.0cm} \bigcup_{t_{0}\leq t<t_{s}}\bigcup_{u_{1}\in U}TR_{\exp((t_{f}-t_{s})f_{q_{1}}(e,u^{o}))}\circ \nnum\\&&\hspace{-.6cm}TR_{\exp((t_{s}-t)f_{0}(e,u^{o}))}\circ TL_{g^{o}(t)} \big(f_{q_{0}}(e,u_{1})-f_{q_{0}}(e,u^{o}(t))\big)\nnum\\&&\hspace{-.60cm}+\hspace{-.0cm} \frac{dt_{s}(\epsilon)}{d\epsilon}|_{\epsilon=0}TR_{\exp((t_{f}-t_{s})f_{q_{1}}(e,u^{o}))}\circ \nnum\\&&\hspace{-.60cm}TL_{g^{o}(t_{s})}\big(f_{q_{1}}(e,u^{o}(t_{s}))- f_{q_{0}}(e,u^{o}(t_{s}))\big) \nnum\\&&\hspace{-.60cm}\subset T_{g^{o}(t_{f})}G,\hspace{.1cm} t\in[t_{0},t_{s}),\EN
and \EQ K_{t_{f}}=K^{1}_{t_{f}}\bigcup K^{2}_{t_{f}}.\EN
\end{lemma}
\begin{proof}
Based on the definition of pull backs (see \cite{Lewis, Abra}), we have
\EQ \langle \mathfrak{I}^{*^{-1}}_{g^{o}(t_{f})}\big(dh(g^{o}(t_{f}))\big),TL_{g^{o^{-1}}(t_{f})}\big(v_{\pi}(t_{f})\big)\rangle=\nnum\\\langle T^{*}L_{g^{o}(t_{f})}\circ \mathfrak{I}^{*^{-1}}_{g^{o}(t_{f})}\big(dh(g^{o}(t_{f}))\big),v_{\pi}(t_{f})\big\rangle,\EN
and since by the definition $\mathfrak{I}^{*}_{g}=T^{*}L_{g^{-1}}$, then 
\EQ &&\langle \mathfrak{I}^{*^{-1}}_{g^{o}(t_{f})}\big(dh(g^{o}(t_{f}))\big),TL_{g^{o^{-1}}(t_{f})}\big(v_{\pi}(t_{f})\big)\rangle=\nnum\\&&\langle dh(g^{o}(t_{f})),v_{\pi}(t_{f})\rangle.\EN

We apply the Taylor expansion on Riemannian manifolds (see \cite{Smith}) to  $h$. To this end, one needs to extend $v_{\pi}(t_{f})\in T_{g(t_{f})}G$ to a smooth vector field denoted by $\tilde{\mathcal{V}}_{\pi}\in \mathfrak{X}(G)$ such that $\tilde{\mathcal{V}}_{\pi}(g(t_{f}))=v_{\pi}(t_{f})$. It is  shown in \cite{Lee2} that this  extension always exists.
Employing the extended smooth vector field $\tilde{\mathcal{V}}_{\pi}\in \mathfrak{X}(G)$, we have
\EQ \label{rrr}h(g(t_{f})\exp(\theta TL_{g^{-1}(t_{f})}v_{\pi}(t_{f}))\hspace{0cm}=\hspace{0cm}h(g(t_{f}))+\nnum\\\theta (\nabla_{\tilde{\mathcal{V}}_{\pi}}h)(g(t_{f}))+\nnum\\o(\theta),  0<\theta< \theta^{*},\nnum\\\EN
where $\mathfrak{X}(G)$ is the space of smooth vector fields on $G$ and $\nabla$ is a linear connection (Cartan-Schouten (0) connection) corresponding to $\exp$ on $G$, see \cite{Mah}.
\\\\
Here we show that $K_{t_{f}}$, defined in Lemma \ref{l6}, contains all the state perturbations at $t_{f}$.
 Lemma \ref{l2} and Proposition \ref{p2} together imply that \\$K^{1}_{t_{f}}=\bigcup_{t_{s}\leq t\leq t_{f}}\bigcup_{u_{1}\in U}T\Phi^{(t_{f},t)}_{f_{q_{1}}}[f_{q_{0}}(g^{o}(t),u_{1})-f_{q_{0}}(g^{o}(t),u^{o}(t))]=\bigcup_{t_{s}\leq t<t_{f}}\bigcup_{u_{1}\in U}TR_{\exp((t_{f}-t)f_{q_{1}}(e,u^{o}))}\circ TL_{g^{o}(t)}\big(f_{q_{1}}(e,u_{1})-f_{q_{1}}(e,u^{o}(t))\big)$ contains all the state perturbations at $g^{o}(t_{f})$ for all the elementary control perturbations inserted after   $t_{s}$. 
For all the control perturbations applied at $t_{0}<t<t_{s}$, either $t_{s}(\epsilon)<t_{s}$ or $t_{s}\leq t_{s}(\epsilon)$, where $t_{s}(\epsilon)$ is the switching time corresponding to $u_{\pi}(t,\epsilon)$.

Following the results of \cite{Taringoo5}, in a local chart around $g(t_{s})$, the differentiability of $t_{s}(\epsilon)$ with respect to $\epsilon$ implies
\EQ&&\frac{d \Phi_{\pi, f_{q_{1}}}^{(t_{s}(\epsilon),t^{1}),g} }{d\epsilon}|_{\epsilon=0}=\hspace{-.1cm} T\Phi^{(t^{-}_{s},t^{1})}_{f_{q_{0}}}[f_{q_{0}}(g^{o}(t^{1}),u_{1})-\nnum\\&&\hspace{-.4cm}f_{q_{0}}(g^{o}(t^{1}),u^{o}(t^{1}))]+(\frac{d t_{s}(\epsilon)}{d \epsilon}|_{\epsilon=0}) \Big(f_{q_{0}}(g^{o}(t^{-}_{s}),u^{o}(t^{-}_{s}))-\nnum\\&&\hspace{-.4cm}f_{q_{1}}(g^{o}(t_{s}),u^{o}(t_{s}))\Big)\in T_{g^{o}(t_{s})}G,\EN
therefore by Lemma \ref{l2}
 \EQ &&\hspace{-0cm}K^{2}_{t_{f}}= \bigcup_{t_{0}< t<t_{s}}\bigcup_{u\in U}\big\{T\Phi^{(t_{f},t_{s})}_{f_{q_{1}}}\circ T\Phi^{(t^{-}_{s},t)}_{f_{q_{0}}}[f_{q_{0}}(x^{o}(t,u_{1}))\nnum\\&&\hspace{.6cm}-f_{q_{0}}(x^{o}(t),u^{o}(t))] +(\frac{dt_{s}(\epsilon)}{d\epsilon}|_{\epsilon=0}).T\Phi^{(t_{f},t_{s})}_{f_{q_{1}}}\times \nnum\\&&\hspace{.6cm}\Big( \big[f_{q_{0}}(x^{o}(t^{-}_{s}),u^{o}(t^{-}_{s}))\big]-f_{q_{1}}(x^{o}(t_{s}),u^{o}(t_{s}))\Big)\big\}\nnum\\&&\hspace{.6cm}=\bigcup_{t_{0}\leq t<t_{s}}\bigcup_{u_{1}\in U}TR_{\exp((t_{f}-t_{s})f_{q_{1}}(e,u^{o}))}\circ \nnum\\&&\hspace{-.6cm}TR_{\exp((t_{s}-t)f_{0}(e,u^{o}))}\circ TL_{g^{o}(t)} \big(f_{q_{0}}(e,u_{1})-f_{q_{0}}(e,u^{o}(t))\big)\nnum\\&&\hspace{-.60cm}+\hspace{-.0cm} \frac{dt_{s}(\epsilon)}{d\epsilon}|_{\epsilon=0}TR_{\exp((t_{f}-t_{s})f_{q_{1}}(e,u^{o}))}\circ \nnum\\&&\hspace{-.60cm}TL_{g^{o}(t_{s})}\big(f_{q_{1}}(e,u^{o}(t_{s}))- f_{q_{0}}(e,u^{o}(t_{s}))\big)\subset T_{g^{o}(t_{f})}G,\nnum\\&& t\in(t_{0},t_{s}),\EN
contains all the state variations at $g^{o}(t_{f})$ corresponding to all elementary control perturbations at $t\in(t_{0},t_{s})$. 
Since $K_{t_{f}}$ contains all the state perturbations at  $g^{o}(t_{f})$,  choosing $v_{\pi}(t_{f})\in K_{t_{f}}\subset T_{x(t_{f})}G$ implies that at least at one particular time, one particular elementary control variation $\big(u_{\pi(t^{1}(v_{\pi}), u_{1}(v_{\pi}))}(t,\epsilon),$ where $u_{1}(v_{\pi})$ is the needle control resulting in the control variation $u_{\pi}(t,\epsilon)$\big) results in the final state variation  $v_{\pi}(t_{f})\in K_{t_{f}}$.
Note that choosing $\epsilon=\theta$,  $h\left(g^{o}(t_{f})\exp(\theta TL_{g^{-1}(t_{f})}v_{\pi}(t_{f})\right)$ and $h(g_{\epsilon}(t_{f}))$, where $g_{\epsilon}(t_{f})$ is the final state  curve obtained with respect to $\epsilon$ variation, are equal to first order since they have  the same first order derivative with respect to $\epsilon$.  By the construction of $u_{\pi}(t,\epsilon)\in \mathcal{U}$, $g_{\epsilon}(t_{f})$ is a curve in the reachable set of the hybrid system at $t_{f}$. The minimality of $x^{o}(t_{f})$ consequently implies that $h(x^{o}(t_{f}))\leq h(x_{\epsilon}(t_{f}))$; then $h(g_{\epsilon}(t_{f}))- h\left(g^{o}(t_{f})\exp(\epsilon TL_{g^{-1}(t_{f})}v_{\pi}(t_{f})\right)=o(\epsilon)$ together with (\ref{rrr}) implies
\EQ\label{rr} 0\leq(\nabla_{\tilde{\mathcal{V}}_{\pi}}h)(g^{o}(t_{f})),\quad \tilde{\mathcal{V}}_{\pi}(g^{o}(t_{f}))=v_{\pi}(t_{f}), \nnum\\\forall v_{\pi}(t_{f})\in K_{t_{f}}.\EN
For the smooth function $h:\mathcal{M}\rightarrow R$, the properties of the linear connections (see \cite{Lee3}) imply 
\EQ \tilde{\mathcal{V}}_{\pi}(h)(g^{o}(t_{f}))&=&\big(\nabla_{\tilde{\mathcal{V}}_{\pi}}h\big)(g^{o}(t_{f}))\nnum\\&=&\sum^{n}_{i=1}v^{i}_{\pi}(t_{f})\frac{\partial h}{\partial g_{i}}(g^{o}(t_{f})),\EN
where the second equality uses local coordinates.
Therefore by the definition of $dh$ we have 
\EQ  \big(\nabla_{\tilde{\mathcal{V}}_{\pi}}h\big)(g^{o}(t_{f}))=\langle dh(g^{o}(t_{f})),v_{\pi}(t_{f})\rangle,\EN
which implies 
\EQ\langle dh(g^{o}(t_{f})),v_{\pi}(t_{f})\rangle\geq 0, \quad\forall v_{\pi}(t_{f})\in K_{t_{f}}, \EN
and completes the proof.
\end{proof}

 The following lemma gives the relation between  $\textbf{G}(g^{o}(t_{s}))(N,.)\in T^{*}_{g^{o}(t_{s})}G $ and any tangent vector $X\in T_{g^{o}(t_{s})}\mathcal{S}\subset T_{g^{o}(t_{s})}G$.
\begin{lemma}
\label{l5}
Consider  an autonomous  HOCP consisting of  two different regimes separated by  a $k$ dimensional  embedded switching manifold $\mathcal{S}\subset G$;  then at the optimal switching state $g^{o}(t_{s})\in \mathcal{S}$  and switching time $t_{s}$ we have 
\EQ\label{vv} &&\langle \mathfrak{I}^{*^{-1}}_{g^{o}(t_{s})}\Big(\textbf{G}(g^{o}(t_{s}))(N,.)\Big),TL_{g^{o^{-1}}(t_{s})}X\rangle=0, \hspace{.2cm}\nnum\\&&\forall X\in T_{g^{o}(t_{s})}\mathcal{S}.\EN
\end{lemma}
\begin{proof}
The proof is immediate since
\EQ  \langle \mathfrak{I}^{*^{-1}}_{g^{o}(t_{s})}\Big(\textbf{G}(g^{o}(t_{s}))(N,.)\Big),TL_{g^{o^{-1}}(t_{s})}X\rangle=\nnum\\ G(g^{o}(t_{s}))(N,X)=0.\EN
\end{proof}

Here we give the proof for the HMP theorem on $G$. For simplicity of notation we simply denote the optimal trajectory by $g(t)\in G$.

      \begin{proof}
       \textit{Step 1}: All the analyses here are performed along the optimal state trajectory $g^{o}(.)$, however, for simplicity of notation the superscript $o$ is omitted for the optimal state trajectory $g^{o}(.)$.
        First consider $t_{s}<t^{1},$ where the needle variation is applied at time $t^{1}$. As shown in \cite{Taringoo5}, we have
        \EQ\label{2}0\leq \langle  dh,v_{\pi}(t_{f})\rangle,\quad \forall v_{\pi}\in K_{t_{f}},\EN
where $dh\in T^{*}_{g(t_{f})}G$.
As mentioned before, the cotangent bundle of the Lie group is identified by $G\times \mathcal{L}$ therefore
\EQ \mathfrak{I}^{*^{-1}}_{g(t_{f})}(dh)\in\mathcal{L}^{*}.\EN
By employing (\ref{2}) and the results of Proposition \ref{p1}, we have
\EQ&&\hspace{-.5cm} 0\leq\langle \mathfrak{I}^{*^{-1}}_{g(t_{f})}(dh),TL_{g^{-1}(t_{f})}\circ TR_{\exp((t_{f}-t^{1})f_{q_{1}}(e,u^{o}))}\circ \nnum\\&&\hspace{-.5cm}  TL_{g(t^{1})}\big(f_{q_{1}}(e,u_{1})-f_{q_{1}}(e,u^{o}(t^{1}))\rangle.\EN
The flow of the left invariant system on $G$ results in
\EQ g^{o}(t_{f})=L_{g(t)}\exp((t_{f}-t)f_{q_{1}}(e,u)),\hspace{.2cm}t\in[t_{s},t_{f}),\EN
then by the vector space properties of $\mathcal{L}$ and one parameter subgroups property of $\exp$ we have
\EQ g(t)=L_{g^{o}(t_{f})}\exp(-(t_{f}-t)f_{q_{1}}(e,u)),\EN  

which by the definition of $\mathfrak{I}^{*}_{g}$ given in (\ref{II}) finally gives 
\EQ&&0\leq \langle T^{*}L_{g(t^{1})}\circ T^{*}R_{\exp((t_{f}-t^{1})f_{q_{1}}(e,u))} (dh),\nnum\\&&f_{q_{1}}(e,u_{1})-f_{q_{1}}(e,u^{o}(t^{1}))\rangle.\EN
Therefore $\forall u\in U$
\EQ &&\hspace{-1cm}\langle  T^{*}L_{g(t^{1})}\circ T^{*}R_{\exp((t_{f}-t^{1})f_{q_{1}}(e,u))} (dh),f_{q_{1}}(e,u^{o}(t^{1}))\rangle\leq\nnum\\ &&\hspace{-1cm}\langle  T^{*}L_{g(t^{1})}\circ T^{*}R_{\exp((t_{f}-t^{1})f_{q_{1}}(e,u))} (dh),f_{q_{1}}(e,u_{1})\rangle,\nnum\\\EN
and 
\EQ  T^{*}L_{g(t^{1})}\circ T^{*}R_{\exp((t_{f}-t^{1})f_{q_{1}}(e,u))} (dh)\in \mathcal{L}^{*}.\EN
The adjoint variable is then defined by 
\EQ&&\lambda(t)=  T^{*}L_{g(t)}\circ T^{*}R_{\exp((t_{f}-t)f_{q_{1}}(e,u))} (dh)\in\mathcal{L}^{*},\nnum\\&& t_{s}\leq t\leq t_{f}.\EN
\textit{Step 2}: Second consider $t_{0}\leq t^{1}< t_{s}$ where $t^{1}$ is the needle variation time. 
For a given switching time $t_{s}$, the differential form of the normal vector  is then given by $\textbf{G}(g(t_{s}))(N,.)\in T^{*}_{g(t_{s})}G$.
Here we have two possibilities, $(i)$: $t_{s}\leq  t_{s}(\epsilon)$ and $(ii)$: $t_{s}(\epsilon)< t_{s}$.
The corresponding control needle variations for these two possibilities are given as follows:
\EQ (i):t_{s}\leq  t_{s}(\epsilon),\quad u_{\pi}(t,\epsilon)=\left\{ \begin{array}{cc} \quad \hspace{-1cm}u^{o}(t) \quad \hspace{.2cm}t\leq t^{1}-\epsilon\\ \hspace{.2cm}u_{1}\quad  \hspace{.6cm}t^{1}-\epsilon\leq t\leq t^{1}\\ \hspace{-.5cm}u^{o}(t) \quad \hspace{.2cm}t^{1}< t \leq t_{s}\\u^{o}(t_{s})\quad t_{s}\leq t < t_{s}(\epsilon) \end{array}\right.,\nnum\EN
and
\EQ (ii):t_{s}(\epsilon)< t_{s},\quad u_{\pi}(t,\epsilon)=\left\{ \begin{array}{cc} \quad \hspace{-1cm}u^{o}(t) \quad \hspace{.2cm}t\leq t^{1}-\epsilon\\ \hspace{.2cm}u_{1}\quad  \hspace{.6cm}t^{1}-\epsilon\leq t\leq t^{1}\\ \hspace{-.0cm}u^{o}(t) \quad \hspace{.3cm}t^{1}< t < t_{s}(\epsilon)\\u^{o}(t_{s})\quad t_{s}(\epsilon)\leq t \leq t_{s} \end{array}\right..\nnum\EN
Notice that $u^{o}(t_{s})$ in $(i)$ corresponds to $f_{q_{0}}$ under the optimal control and in $(ii)$ corresponds to $f_{q_{1}}$ under the optimal control. Following Lemmas 4.2 and 4.3 in \cite{Taringoo5},  in the case $t_{s}(\epsilon)<t_{s}=t^{o}_{s}$, we have 
\EQ\label{llq} &&\hspace{-.7cm}\frac{d g_{\pi, f_{q_{0}}}^{(t_{s}(\epsilon),t^{1}),g(t^{1}))} }{d\epsilon}|_{\epsilon=0}=(\frac{d t_{s}(\epsilon)}{d \epsilon}|_{\epsilon=0}) TL_{g(t_{s})}\big(f_{q_{0}}(e,u^{o}(t_{s}))\big)+\nnum\\&&\hspace{-.7cm}TR_{\exp(t_{s}-t^{1})f_{q_{0}}(e,u^{o})}\circ TL_{g(t^{1})}\big(f_{q_{0}}(e,u)-f_{q_{0}}(e,u^{o}(t^{1}))\big)\nnum\\&&\subset T_{g(t_{s})}G.\EN
And for  the case in which $t_{s}<t_{s}(\epsilon) $ we have 
\EQ\label{llqq} &&\hspace{-.7cm}\frac{d g_{\pi, f_{q_{0}}}^{(t_{s}(\epsilon),t^{1}),g(t^{1}))} }{d\epsilon}|_{\epsilon=0}=-(\frac{d t_{s}(\epsilon)}{d \epsilon}|_{\epsilon=0}) TL_{g(t_{s})}\big(f_{q_{0}}(e,u^{o}(t_{s}))\big)\nnum\\&&\hspace{-.7cm}+TR_{\exp(t_{s}-t^{1})f_{q_{0}}(e,u^{o})}\circ TL_{g(t^{1})}\big(f_{q_{0}}(e,u)-f_{q_{0}}(e,u^{o}(t^{1}))\big)\nnum\\&&\subset T_{g(t_{s})}G.\EN
The differentiability of $t_{s}(.)$ with respect to $\epsilon$ is established in \cite{Taringoo5}, Lemma 4.2.
Equation (\ref{vv}) together with Lemma \ref{l2} implies
\EQ \label{las}&&\hspace{-.5cm}\frac{d t_{s}(\epsilon)}{d \epsilon}|_{\epsilon=0}=-\langle \mathfrak{I}^{*^{-1}}_{g(t_{s})}\Big(\textbf{G}(g(t_{s}))(N,.)\Big),f_{q_{0}}(e,u^{o}(t_{s}))\rangle^{-1}\nnum\\&&\hspace{-.5cm}+\langle T^{*}L_{g(t^{1})}\circ T^{*}R_{\exp((t_{s}-t^{1})f(e,u))} \nnum\\&&\big(\textbf{G}(g(t_{s}))(N,.)\big),f_{q_{0}}(e,u_{1})-f_{q_{0}}(e,u^{o}(t^{1}))\rangle,\EN
since $g_{\pi, f_{q_{0}}}^{(t_{s}(.),t^{1},g(t^{1}))}:[0,\epsilon]\rightarrow \mathcal{S}$ and $\frac{d g_{\pi, f_{q_{0}}}^{(t_{s}(\epsilon),t^{1}),g(t^{1}))} }{d\epsilon}|_{\epsilon=0}\in T_{g(t_{s})}\mathcal{S}\subset  T_{g(t_{s})}G $.
In the second case
\EQ \label{las1}&&\hspace{-.5cm}\frac{d t_{s}(\epsilon)}{d \epsilon}|_{\epsilon=0}=\langle \mathfrak{I}^{*^{-1}}_{g(t_{s})}\big(\textbf{G}(g(t_{s}))(N,.)\big),f_{q_{0}}(e,u^{o}(t_{s}))\rangle^{-1}\times\nnum\\&&\hspace{-.5cm}\langle T^{*}L_{g(t^{1})}\circ T^{*}R_{\exp((t_{s}-t^{1})f(e,u))} \big(\textbf{G}(g(t_{s}))(N,.)\big),\nnum\\&&f_{q_{0}}(e,u_{1})-f_{q_{0}}(e,u^{o}(t^{1}))\rangle.\EN
In order to obtain the state variation at $t_{s}$, in case (ii), we use the push-forward of the combination of the flows before and after $t_{s}$ as follows:
\EQ\label{kk1}&&\hspace{-.7cm}\frac{dg_{\pi, f_{q_{1}}}^{(t_{s},t_{s}(\epsilon))} \circ g_{\pi, f_{q_{0}}}^{(t_{s}(\epsilon),t^{1},g(t^{1}))} }{d\epsilon}|_{\epsilon=0}=TR_{\exp(t_{s}-t^{1})f_{q_{0}}(e,u^{o})}\nnum\\&&\hspace{-.7cm}\circ TL_{g(t^{1})}\big( f_{q_{0}}(e,u_{1})-f_{q_{0}}(e,u^{o}(t^{1}))\big)\nnum\\&&\hspace{-.7cm}+\frac{d t_{s}(\epsilon)}{d \epsilon}|_{\epsilon=0}TL_{g(t_{s})}\big( f_{q_{0}}(e,u^{o}(t^{1}))-f_{q_{1}}(e,u^{o}(t^{1}))\big)\nnum\\&&\in T_{g(t_{s})}G,\EN
and for case (i)
 \EQ\label{kk2}&&\hspace{-.7cm}\frac{d g_{\pi, f_{q_{1}}}^{(t_{s}(\epsilon),t_{s})} \circ g_{\pi, f_{q_{0}}}^{(t_{s},t^{1},g(t^{1}))}}{d\epsilon}|_{\epsilon=0}=TR_{\exp(t_{s}-t^{1})f_{q_{0}}(e,u^{o})}\nnum\\&&\hspace{-.7cm}\circ TL_{g(t^{1})}\big( f_{q_{0}}(e,u_{1})-f_{q_{0}}(e,u^{o}(t^{1}))\big)+\frac{d t_{s}(\epsilon)}{d \epsilon}|_{\epsilon=0}\nnum\\&&\hspace{-.7cm}\times TL_{g(t_{s})}\big(f_{q_{1}}(e,u^{o}(t^{1}))- f_{q_{0}}(e,u^{o}(t^{1}))\big)\in T_{g(t_{s})}G,\nnum\\\EN
 where the differentiability of $t_{s}(\epsilon)$ with respect to $\epsilon$ for hybrid systems on Riemannian manifolds is established in \cite{Taringoo5}. 
The final state variation at the final time $t_{f}$ is now given as follows:

\EQ\label{kj}\hspace{.5cm}&& \frac{d g_{\pi, f_{q_{1}}}^{(t_{f},t^{1},g(t^{1}))}(\epsilon)}{d \epsilon}|_{\epsilon=0}=TR_{\exp(t_{f}-t_{s})f_{q_{1}}(e,u^{o})}\times \nnum\\&&\frac{d g_{\pi, f_{q_{1}}}^{(t_{s}(\epsilon),t_{s})} \circ g_{\pi, f_{q_{0}}}^{(t_{s},t^{1},g(t^{1}))}}{d\epsilon}|_{\epsilon=0}.\EN

Therefore 
\EQ \label{las2}&&\hspace{-0cm} 0\leq\langle dh(g(t_{f})),TR_{\exp(t_{f}-t_{s})f_{q_{1}}(e,u^{o})} \nnum\\&&\times\big[ \frac{d t_{s}(\epsilon)}{d \epsilon}|_{\epsilon=0}TL_{g(t_{s})}\big(f_{q_{1}}(e,u^{o}(t_{s}))- f_{q_{0}}(e,u^{o}(t_{s}))\big)
\nnum\\&&+TR_{\exp(t_{s}-t^{1})f_{q_{0}}(e,u^{o})}\circ TL_{g(t^{1})}\big(f_{q_{0}}(e,u_{1})-\nnum\\&& f_{q_{0}}(e,u^{o}(t^{1}))\big)\big]\rangle\hspace{.1cm},
\EN
Hence
\EQ \label{las22}&&\hspace{-0cm} 0\leq \langle dh(g(t_{f})),TR_{\exp(t_{f}-t_{s})f_{q_{1}}(e,u^{o})} \nnum\\&&\times\big[ -\langle \mathfrak{I}^{*^{-1}}_{g(t_{s})}\Big(\textbf{G}(g(t_{s}))(N,.)\Big),f_{q_{0}}(e,u^{o}(t_{s}))\rangle^{-1}\nnum\\&&\hspace{-0cm}\times\langle T^{*}L_{g(t^{1})}\circ T^{*}R_{\exp((t_{s}-t^{1})f(e,u))} \Big(\textbf{G}(g(t_{s}))(N,.)\Big),\nnum\\&&f_{q_{0}}(e,u_{1})-f_{q_{0}}(e,u^{o}(t^{1}))\rangle\nnum\\&&\times TL_{g(t_{s})}\big(f_{q_{1}}(e,u^{o}(t_{s}))- f_{q_{0}}(e,u^{o}(t_{s}))\big)
\nnum\\&&+TR_{\exp(t_{s}-t^{1})f_{q_{0}}(e,u^{o})}\circ TL_{g(t^{1})}\big(f_{q_{0}}(e,u_{1})-\nnum\\&& f_{q_{0}}(e,u^{o}(t^{1}))\big)\big]\rangle\hspace{.1cm},
\EN
equivalently
\EQ&&\hspace{-0cm}0\leq -\langle \mathfrak{I}^{*^{-1}}_{g(t_{s})}\Big(\textbf{G}(g(t_{s}))(N,.)\Big),f_{q_{0}}(e,u^{o}(t_{s}))\rangle^{-1}\nnum\\&&\hspace{-0cm}\times\langle T^{*}L_{g(t^{1})}\circ T^{*}R_{\exp((t_{s}-t^{1})f(e,u))} \Big(\textbf{G}(g(t_{s}))(N,.)\Big),\nnum\\&&f_{q_{0}}(e,u_{1})-f_{q_{0}}(e,u^{o}(t^{1}))\rangle \nnum\\&&\hspace{-0cm}\times \langle dh(g^{o}(t_{f})), TR_{\exp(t_{f}-t_{s})f_{q_{1}}(e,u^{o})}\circ TL_{g(t_{s})}\nnum\\&&\big(f_{q_{1}}(e,u^{o}(t_{s}))- f_{q_{0}}(e,u^{o}(t_{s}))\big)\rangle\nnum\\&&\hspace{-0cm}+\langle dh(g(t_{f})),TR_{\exp(t_{f}-t_{s})f_{q_{1}}(e,u^{o})}\nnum\\&&\circ TR_{\exp(t_{s}-t^{1})f_{q_{0}}(e,u^{o})}\circ TL_{g(t^{1})}\big(f_{q_{0}}(e,u_{1})-\nnum\\&& f_{q_{0}}(e,u^{o}(t^{1}))\big)\rangle,\EN
Let us denote $\mu $ by 
\EQ &&\hspace{-.5cm}\mu=-\langle \mathfrak{I}^{*^{-1}}_{g(t_{s})}\Big(\textbf{G}(g(t_{s}))(N,.)\Big),f_{q_{0}}(e,u^{o}(t_{s}))\rangle^{-1} \nnum\\&& \langle dh(g^{o}(t_{f})), TR_{\exp(t_{f}-t_{s})f_{q_{1}}(e,u^{o})}\nnum\\&& TL_{g(t_{s})}\big(f_{q_{1}}(e,u^{o}(t_{s}))- f_{q_{0}}(e,u^{o}(t_{s}))\big)\rangle,\EN
therefore
\EQ \label{ww}&&\hspace{-.5cm}0\leq\langle dh(g(t_{f})),TR_{\exp(t_{f}-t_{s})f_{q_{1}}(e,u^{o})}\circ\nnum\\&& TR_{\exp(t_{s}-t^{1})f_{q_{0}}(e,u^{o})}\circ TL_{g(t^{1})}\nnum\\&&\hspace{-.4cm}\big(f_{q_{0}}(e,u_{1})- f_{q_{0}}(e,u^{o}(t^{1}))\big)\rangle+\mu\langle T^{*}L_{g(t^{1})}\nnum\\&&\circ T^{*}R_{\exp((t_{s}-t^{1})f(e,u))} \Big(\textbf{G}(g(t_{s}))(N,.))\Big),f_{q_{0}}(e,u_{1})-\nnum\\&&f_{q_{0}}(e,u^{o}(t^{1}))\rangle.\EN
Similar to step 1 we have

\EQ \label{www}&&\hspace{0cm} \langle dh(g(t_{f})),TR_{\exp(t_{f}-t_{s})f_{q_{1}}(e,u^{o})}\circ\nnum\\&& TR_{\exp(t_{s}-t^{1})f_{q_{0}}(e,u^{o})}\circ TL_{g(t^{1})}\nnum\\&&\hspace{-.4cm}\big(f_{q_{0}}(e,u_{1})- f_{q_{0}}(e,u^{o}(t^{1}))\big)\rangle=\nnum\\&&\langle T^{*}L_{g(t^{1})}\circ T^{*}R_{\exp(t_{s}-t^{1})f_{q_{0}}(e,u^{o})}\nnum\\&&\circ T^{*}R_{\exp(t_{f}-t_{s})f_{q_{1}}(e,u^{o})}\big(dh(g(t_{f}))\big),\big(f_{q_{0}}(e,u_{1})-\nnum\\&& f_{q_{0}}(e,u^{o}(t^{1}))\big)\rangle,\EN

Combining (\ref{ww}) and (\ref{www}) we have
\EQ &&0\leq \langle T^{*}L_{g(t^{1})}\circ T^{*}R_{\exp(t_{s}-t^{1})f_{q_{0}}(e,u^{o})}\circ\nnum\\&& T^{*}R_{\exp(t_{f}-t_{s})f_{q_{1}}(e,u^{o})}\big(dh(g(t_{f}))\big),\nnum\\&&\hspace{-.4cm}\big(f_{q_{0}}(e,u_{1})- f_{q_{0}}(e,u^{o}(t^{1}))\big)\rangle+\mu\langle T^{*}L_{g(t^{1})}\nnum\\&&\circ T^{*}R_{\exp((t_{s}-t^{1})f(e,u))} \Big(\textbf{G}(g(t_{s}))(N,.)\Big),f_{q_{0}}(e,u_{1})-\nnum\\&&f_{q_{0}}(e,u^{o}(t^{1}))\rangle.\EN
The adjoint process $\lambda$ is defined as follows:
\EQ &&\hspace{-.5cm}\lambda(t)= T^{*}L_{g(t)}\circ T^{*}R_{\exp(t_{s}-t)f_{q_{0}}(e,u^{o})}\circ\nnum\\&& T^{*}R_{\exp(t_{f}-t_{s})f_{q_{1}}(e,u^{o})}\big(dh(g(t_{f}))\big)\nnum\\&&\hspace{-.5cm}+\mu T^{*}L_{g(t)}\circ T^{*}R_{\exp((t_{s}-t)f(e,u))} \Big(\textbf{G}(g(t_{s}))(N,.)\Big).\nnum\\\EN
At time $t=t_{s}$ we have 
\EQ \lambda(t^{-}_{s})=\lambda(t_{s})+\mu T^{*}L_{g(t_{s})}(\textbf{G}(g(t_{s}))(N,.))\in \mathcal{L}^{*}.\EN
The proof for the continuity of Hamiltonians follows from the results of \cite{Taringoo5}.
It remains to show 
\EQ \label{sta11}&&\frac{dg}{dt}=TL_{g(t)}\big(\frac{\partial H_{q_{i}}}{\partial \lambda}\big),\quad  \frac{d\lambda}{dt}=-(ad)^{*}_{\frac{\partial H_{q_{i}}}{\partial \lambda}}(\lambda(t)),\nnum\\&&t\in[t_{i},t_{i+1}), q_{i}\in Q.\EN
The first part of (\ref{sta11}) is obvious by the definition of $H_{q_{i}}:=\langle \lambda,TL_{g^{-1}(t)}f_{q_{i}}(g(t),u)\rangle$, since $f_{q_{i}}$ is left invariant and  $\frac{dg}{dt}=TL_{g(t)}\circ TL_{g^{-1}(t)}f_{q_{i}}(g(t),u)=f_{q_{i}}(g(t),u)$.
To prove the second relation in (\ref{sta11}), for a given $g\in G$, we employ the \textit{conjugate map} $I_{g}:G\rightarrow G$ which is given as follows (see \cite{Abra, Lewis}):
\EQ I_{g}(h)=g\star h\star g^{-1}.\EN
The \textit{adjoint map} $Ad_{g}:\mathcal{L}\rightarrow \mathcal{L}$ is defined by
\EQ Ad_{g}=TI_{g}=TL_{g}\circ TR_{g^{-1}},\EN
where the dual of the adjoint map $Ad^{*}_{g}$ is calculated as $Ad^{*}_{g}=T^{*}L_{g}\circ T^{*}R_{g^{-1}}$.  
As is obtained in step 1, $\lambda(t)=  -T^{*}L_{g(t)}\circ T^{*}R_{\exp((t_{f}-t)f(e,u))} (dh)\in\mathcal{L}^{*},$ then in order to establish the second relation, in (\ref{sta11}), it is enough to show that $\lambda(t)=Ad^{*}_{g(t)}(\lambda(0))$ where without loss of generality we set $t_{s}=0$ and $\lambda(0)=\lambda(t_{s})$ (see \cite{Jurd}, Theorem 5, Chapter 12). Therefore, to prove $\lambda(t)=Ad^{*}_{g(t)}(\lambda(0))$, one needs to show 
\EQ \label{sho}&&\hspace{-.7cm}T^{*}L_{g(t)}\circ T^{*}R_{\exp((t_{f}-t)f(e,u))} (dh)=\nnum\\&&\hspace{-.7cm}T^{*}L_{g(t)}\circ T^{*}R_{g^{-1}(t)}\circ T^{*}L_{g(0)}\circ T^{*}R_{\exp((t_{f})f(e,u))} (dh).\nnum\\\EN
Employing the group operation we have
\EQ g(t_{f})&=&g(0)\star \exp(t_{f}f_{q_{1}}(e,u^{o}))\nnum\\&=& g(0)\star g(t)\star g^{-1}(t)\star \exp(t_{f}f_{q_{1}}(e,u^{o})),\nnum\\\EN
and also
\EQ g(t_{f})=g(t)\star \exp((t_{f}-t)f_{q_{1}}(e,u^{o})),\EN
then
\EQ &&\hspace{0cm}R_{\exp((t_{f}-t)f_{q_{1}}(e,u^{o}))}(g(t))= \nnum\\&&R_{ \exp(t_{f}f_{q_{1}}(e,u^{o}))}\circ L_{g(0)}\circ R_{g^{-1}(t)}(g(t)),\quad \forall g(t)\in G,\nnum\\\EN
which implies
\EQ &&T^{*}R_{\exp((t_{f}-t)f_{q_{1}}(e,u^{o}))}=\nnum\\&&T^{*}R_{g^{-1}(t)}\circ T^{*}L_{g(0)}\circ T^{*}R_{\exp(t_{f}f_{q_{1}}(e,u^{o}))},\EN
which shows (\ref{sho}). As shown in \cite{Jurd}, Theorem 5, Chapter 12, differentiation of $\lambda(t)=Ad^{*}_{g(t)}(\lambda(0))$ with respect to $t$ implies 
\EQ \frac{d \lambda}{dt}=-\big((ad)^{*}\frac{dg}{dt}\big)\lambda(t)=-\big((ad)^{*}\frac{\partial H_{i}}{\partial \lambda}\big)\lambda(t),\EN
and this completes the proof. The analogous argument holds for $\lambda(t),\quad t_{0}\leq t< t_{s}$.
\end{proof}
We note that in the case of controlled switching hybrid systems, the adjoint variable $\lambda^{o}$ is continuous at the optimal switching time, i.e. (\ref{kirkir}) changes to $\lambda^{o}(t^{-}_{s})=\lambda^{o}(t_{s})$.  
\subsection{Interior Optimal Switching States}

Here we specify a hypothesis for MHOCP which expresses the HMP statement based on a differential form of the hybrid value function.\\\\
\textbf{\textit{A4}}: For an MHOCP, the value function $v(g,t),\hspace{.2cm} g\in G,t\in (t_{0},t_{f})$, is assumed to be differentiable at the optimal switching state $g^{o}(t^{-}_{s})$ in the switching manifold $\mathcal{S},$ where the optimal switching state is an interior point of the attainable switching states on the switching manifold.\\\\
We note that $\textbf{\textit{A4}}$ rules out MHOCPs derived from BHOCPs (see Lemma \ref{l3}).
The following theorem gives the HMP statement for an accessible  MHOCP satisfying \textbf{\textit{A4}}.

\begin{theorem}
\label{t2}
Consider a hybrid system satisfying the hypotheses presented in \textbf{\textit{A1-A4}} on a Lie group $G$  and an  embedded switching submanifold $\mathcal{S}\subset G$. Then  corresponding to  the  optimal  control and optimal state trajectory $u^{o}(t),g^{o}(t)$,  there exists a nontrivial $\lambda^{o}\in \mathcal{L}^{*}$ along the optimal state trajectory such that:
\EQ H_{q_{i}}(g^{o}(t),\lambda^{o}(t),u^{o}(t))\leq H_{q_{i}}(g^{o}(t),\lambda^{o}(t),u_{1}),\nnum\\ \forall u_{1}\in U, t\in[t_{0},t_{f}],i=0,1, \EN
and at the optimal switching state and switching time $g^{o}({t_{s}}),t_{s}$ we have 

\EQ \lambda^{o}(t^{-}_{s})=\lambda^{o}(t_{s})+\mu T^{*}L_{g(t_{s})}(dv(g^{o}_{s},t_{s}))\in \mathcal{L}^{*}.\EN
and the continuity of the Hamiltonian is given as follows
\EQ\hspace{-.5cm} &&H_{q_{0}}(g^{o}(t^{-}_{s}),\lambda^{o}(t^{-}_{s}),u^{o}(t^{-}_{s}))=\nnum\\&&H_{q_{1}}(g^{o}(t_{s}),\lambda^{o}(t_{s}),u^{o}(t_{s})).\EN
The adjoint variable $\lambda$ satisfies
 \EQ \frac{dg}{dt}=T_{e}L_{g(t)}\big(\frac{\partial H_{q_{i}}}{\partial \lambda}\big),\quad\label{lam1}  \frac{d\lambda}{dt}=-(ad)^{*}_{\frac{\partial H_{q_{i}}}{\partial \lambda}}(\lambda(t)),\nnum\\ t\in[t_{i},t_{i+1}), q_{i}\in Q,\EN
 where 
 \EQ H_{q_{i}}(g,\lambda,u):=\langle \lambda,TL_{g^{-1}}f_{q_{i}}(g,u)\rangle,\EN
and
\EQ dv(g^{o}(t^{-}_{s}),t_{s})=\frac{\partial v(g^{o}_{s},t^{-}_{s})}{\partial g}\in T^{*}_{g^{o}(t^{-}_{s})}G.\EN
\end{theorem}
\begin{proof}
The proof is a repetition of the proof of Theorem \ref{t22} where $\textbf{G}(g)(N,.)$  is replaced by  $dv(g,t)$, where $\textbf{G}$ is the Riemannian metric associated with the inner product $\textbf{I}$, see Lemma \ref{l10}.  As shown in Theorem \ref{t22}, the adjoint process discontinuity is given by 
\EQ \lambda^{o}(t^{-}_{s})&=&\lambda^{o}(t_{s})+\mu T^{*}L_{g(t_{s})}dv(g(t_{s}),t_{s}).\EN
\end{proof}

\section{Exp-Gradient HMP Algorithm}
In this section we introduce an algorithm which is based upon the HMP algorithm first introduced in \cite{Shaikh} and then extended on Riemannian manifolds in \cite{Taringoo8}. The algorithm presented in \cite{Taringoo8} is an extension of the Steepest Descent Algorithm along the geodesics on Riemannian manifolds. As known (see \cite{Jost}), geodesics are defined as length minimizing curves on Riemannian manifolds. The solution of the Euler-Lagrange variational problem associated with the length minimizing problem shows that  all the geodesics on $\mathcal{M}$ connecting $\gamma(a),\gamma(b)\in \mathcal{M}$ must satisfy the following system of ordinary differential equations:
\EQ \label{geo}\ddot{x_{i}}(s)+\sum^{n}_{j,k=1}\Gamma^{i}_{j,k}\dot{x_{j}}(s)\dot{x_{k}}(s)=0,\quad i=1,...,n,\EN
where 
\EQ\label{cris} \Gamma^{i}_{j,k}=\frac{1}{2}\sum^{n}_{l=1}g^{il}(g^{\mathcal{M}}_{jl,k}+g^{\mathcal{M}}_{kl,j}-g^{\mathcal{M}}_{jk,l}),\quad g^{\mathcal{M}}_{jl,k}=\frac{\partial}{\partial x_{k}}g^{\mathcal{M}}_{jl},\nnum\\\EN
where $g^{\mathcal{M}}$ is the Riemannain metric corresponding to $\mathcal{M}$ and
 all the indices $i,j,k$ here run from $1$ up to $n=dim(\mathcal{M})$ and $[g^{ij}]=[g^{\mathcal{M}}_{ij}]^{-1}$.
 
 In order to introduce the gradient of the value function on a Lie group $G$ we employ the notion of inner product on a finite dimensional Lie algebra $\mathcal{L}$ defined in Section \ref{s2}.
 For a given value function $v:G\rightarrow \mathds{R}$ on a Lie group $G$ we have 
 \EQ dv|_{g}:=\frac{\partial v}{\partial g}\in T^{*}_{g}G.\EN
 The gradient of $v$, i.e. $\nabla v$, is defined by
 \EQ \langle dv, X_{g}\rangle=\textbf{G}(g)(\nabla v,X_{g}),\hspace{.2cm} \forall X_{g}\in T_{g}G,\EN
 which can be written as
 \EQ \label{ak} \langle dv, X_{g}\rangle&\hspace{-.1cm}=&\hspace{-.1cm}\langle dv, TL_{g}X\rangle=\langle T^{*}L_{g}dv, X\rangle\nnum\\&\hspace{-.1cm}=&\hspace{-.1cm}\textbf{I}(TL_{g^{-1}}\nabla v,TL_{g^{-1}}X_{g}),\hspace{.2cm}\forall X_{g}\in T_{g}G.\nnum
 \\\EN
 We call $TL_{g^{-1}}\nabla v$ the \textit{projected gradient} of $v$ on $\mathcal{L}$.
 Similar to the geodesic gradient flow defined on Riemannian manifold $\mathcal{M}$ in \cite{Taringoo8}, we introduce \textit{Exp-Gradient Flow } on Lie groups as follows:
   \begin{definition}\label{d4} (\textit{Exp-Gradient Flow}) 
  Let $\theta^{0}=0$, and $g(\theta^{0})=g^{0}\in G $, then for all $0 \leq k$ and all $g^{k}$ such that $TL_{g^{k^{-1}}}\nabla v(g^{k})\neq 0$, define 
  \EQ\label{13}&& \gamma_{g^{k}}(\theta)=g^{k}(\theta):=\hspace{0cm}g^{k}\star \exp(-\theta TL_{g^{k^{-1}}}\nabla v(g^{k})), \nnum\\ &&\theta\in[\theta^{k},\theta^{k+1}),\quad g(\theta)\in G,\EN
  where
  \EQ \label{114}&&g^{k+1}=g^{k}(\theta^{k+1}),\hspace{.2cm}\theta^{k+1}=\sup\limits_{s}\{s; \frac{dv(g(t))}{dt}\leq 0,\nnum\\&& t\in [\theta^{k},s), s\in[\theta^{k},\theta^{k}+1)\}.\EN\end{definition}
  
Over the interval of existence $[0,\omega)$ we denote the total flow induced by (\ref{13})) as
\EQ \label{3333} \varphi(\theta,g^{0})=\Pi^{n}_{i=1}\psi_{i}(\theta^{i-1},\theta^{i},g^{i-1})\circ\psi_{n}(\theta^{n},\theta,g^{n}),\EN
where  \EQ \psi_{i}(\theta^{i-1},\theta^{i},g^{i-1})=\gamma_{g^{i-1}}(\theta^{i}-\theta^{i-1}),\nnum\\\hspace{2cm}\gamma_{g^{0}}(\theta^{1}-\theta^{0})=\gamma_{g^{0}}(\theta^{1}),\theta^{0}=0,\EN
 $\theta^{i}-\theta^{i-1}$ is the elapsed time between the switching times $\theta^{i}$, $\theta^{i-1}$ to the next iteration and $n$ is the index number of the last switching  before the instant $\theta$.
By  the continuity of geodesic flows $\{\psi_{i}, 1\leq i<\infty\}$, $\varphi$ is a continuous map on $[0,\omega)$.
In the notation of topological dynamics, and in particular LaSalle Theory (see e.g. \cite{peter,LaSalle}), the limit set of the initial state $x^{0}$ is denoted as $\Omega(g^{0})$, where 
\EQ \label{444} y\in \Omega (g^{0})\Rightarrow \exists \theta_{n},n\geq 1,\quad s.t. \quad \lim_{n\rightarrow \infty}g(\theta_{n})=y,\nnum\\\EN 
when $ \lim_{n\rightarrow \infty}(\theta_{n})= \omega$.
Note the sequence $\{\theta_{n}\}$ is in general distinct from $\{\theta^{n}\}$.\\

\textbf{\textit{H1}}: There exists  $0<b<\infty$ such that the associated sublevel set $ \mathcal{N}_{b}=\{g\in G;\quad  v(g)< b\}$ is (i) open (ii) connected, (iii) contains a strict local minimum $g_{*}$ which is the only local minimum in $\mathcal{N}_{b}$, (iv) $\mathcal{N}_{b}$ has compact closure and (v) $ \mathcal{N}_{g_{*}}\subset\mathcal{N}_{b}$.\\

Without loss of generality, we assume $ \mathcal{N}_{g_{*}}\subset \mathcal{N}_{b-\epsilon}$ for some $\epsilon>0$, then by selecting $g^{0}\in \mathcal{N}_{g_{*}}\subset \mathcal{N}_{b-\epsilon}\subset \mathcal{N}_{b}$ we prove $\omega= \infty$ by the following lemma:
\begin{lemma}
For an initial state $g^{0}\in \mathcal{N}_{g_{*}}$, the existence interval of the flow defined in (\ref{3333}) goes to $\infty$.
\end{lemma}
\begin{proof}
 By \textbf{\textit{H1}} we have $\mathcal{N}_{g_{*}}\subset \mathcal{N}_{b-\epsilon}$. Choose $0<\theta<\theta^{'}$ then if $\theta$ is not a switching time  by the construction of  $\phi$, i.e. (\ref{114}) 
 \EQ v(\varphi(\theta^{'},g^{0}))\leq v(\varphi(\theta,g^{0}))\leq v(g^{0})<b-\epsilon<b.\EN
 We need to prove the statement above when $\theta $ is a switching time. 
 The derivative from the right of the flow $\varphi$ which is the combination of the flows defined in (\ref{13}) at the switching state $g^{k}$ is given by
\EQ\label{sag1} \frac{dv(g^{k}(\theta))}{d\theta}|_{\theta=0}&\hspace{-.1cm}=&\hspace{-.1cm}\langle dv, -TL_{g^{k}}TL_{g^{k^{-1}}}\nabla v\rangle\nnum\\&\hspace{-.1cm}=&\hspace{-.1cm}-\langle T^{*}L_{g^{k}}dv, TL_{g^{k^{-1}}}\nabla v\rangle\nnum\\&\hspace{-.1cm}=&\hspace{-.1cm}-\textbf{I}(TL_{g^{k^{-1}}}\nabla v,TL_{g^{k^{-1}}}\nabla v)<0.\nnum\\\EN

It follows by the construction of $\varphi$ in \ref{3333}, for all $0<\theta<\theta^{'}$, that
\EQ v(\varphi(\theta^{'},g^{0}))\leq v(\varphi(\theta,g^{0}))\leq v(g^{0})<b-\epsilon<b,\EN
and hence for $ \Phi^{+}:=\{\varphi(\theta,g^{0}); 0\leq \theta<\omega\}$
\EQ\hspace{3cm} \overline{\Phi^{+}}\subset\overline{\mathcal{N}}_{b-\epsilon} \subset \mathcal{N}_{b}.\EN

So the flow $\varphi$ is defined everywhere in $\overline{\mathcal{N}_{b-\epsilon}}$, where $\mathcal{N}_{b}$ has compact closure. Hence for all $g\in \mathcal{N}_{b-\epsilon}$ we have an extension of $\varphi$ in $\mathcal{N}_{b}$, therefore the maximum interval of existence of $\varphi(.,g^{0})$ in $\mathcal{N}_{b}$ is infinite. 
 \end{proof}

 \begin{theorem}
   \label{t222}
 Subject to the hypothesis \textbf{\textit{H1}} on $\mathcal{N}_{b}$ and with an initial state $g^{0}$ such that $g^{0}\in \mathcal{N}_{b-\epsilon}\subset M$, $0<\epsilon<b$, either the Geodesic-Gradient flow, $\varphi$, reaches an equilibrium after a finite number of switchings, or it satisfies 
\EQ\hspace{1cm} \varphi(\theta,g^{0})\rightarrow \Omega(g^{0})\subset v^{-1}(c),\quad c\in \mathds{R},\EN
as $\theta\rightarrow \infty$, for some $c\in \mathds{R}$, where 
\EQ\hspace{2cm} \label{stai11}\forall y\in \Omega(g^{0}),\quad \frac{dv(y)}{d\theta}|_{\theta=0}=0,\EN
and, furthermore, the switching  sequence $\{g\}^{\infty}_{0}=\{g^{0},g^{1},\cdots,\}$ converges to the limit point $g_{*}\in \Omega(g^{0})\subset \mathcal{N}_{b}$, where $g_{*}$ is the unique element of $\mathcal{N}_{b}$ such that $\nabla^{\gamma}_{M}v(g_{*})=0$.
  \end{theorem}
  
\begin{proof}
The first statement of the theorem is immediate by the Definition \ref{d4}.
To prove the second statement, similar to the proof of the LaSalle Theorem, we proceed by showing that $v(.)$ is constant on the set $\Omega(x^{0})$.
The precompactness of $\Phi^{+}$ ( that is to say (i): $\overline{\Phi^{+}}\subset \overline{\mathcal{N}}_{b}$, (ii): there does not exist $\theta_{i}\rightarrow \omega, i\rightarrow \infty$, such that $\varphi(\theta_{i},g^{0})\rightarrow \partial \mathcal{N}_{b}$, i.e. $\overline{\Phi^{+}}\bigcap \partial \mathcal{N}_{b}=\oslash$), implies $\Omega(g^{0})\neq \oslash$, see \cite{peter}. By the definition of $\Omega(g^{0})$ we have
\EQ \forall &&y\in \Omega(g^{0})\Rightarrow \exists \theta_{n}, n\geq1, \quad s.t.,\nnum\\&& \varphi(\theta_{n},g^{0})\rightarrow y, \theta_{n}\rightarrow \infty,\EN
and since $v(.)\in C^{1}$, 
\EQ \lim_{n\rightarrow \infty}v(g(\theta_{n}))=\lim_{n\rightarrow \infty}v(\varphi(\theta_{n},g^{0}))=v(y)=:c.\nnum\\\EN
Now choose $y^{'}\in\Omega(g^{0}), y^{'}\neq y, $ then by the existence of a convergent sequence $g(\theta^{'}_{n})$ to $y^{'}$ we have
\EQ\label{4} &&\forall\epsilon>0 \Rightarrow  \exists n,n_{i},k\quad s.t. \quad \theta_n<\theta^{'}_{n_{i}}<\theta_{n+k}\nnum\\&& c-\epsilon<v(g(\theta_{n+k}))\leq v(g(\theta^{'}_{n_{i}})) \leq v(g(\theta_{n}))<c+\epsilon,\nnum\\\EN
i.e. $v(y^{'})=c$, hence $\Omega(x^{0})\subset v^{-1}(c).$
To prove stationarity, i.e. (\ref{stai11}), we observe that $\Omega(x^{0})$ is positive invariant under the flow $\varphi$, i.e. 
\EQ\hspace{2cm} \label{15}\varphi(\theta,\Omega(g^{0}))\subset \Omega(g^{0}),\quad \theta>0.\EN
This follows from the continuity of $\varphi(.,.)$, see \cite{peter}. Differentiability from the right for all $g\in \varphi(\theta,g^{0}), 0<\theta$, implies  
\EQ\label{yy} \frac{dv}{d\theta}|_{\theta=0}&=&\lim_{\theta\rightarrow 0^{+}}\frac{v(\varphi(\theta,y))-v(\varphi(0,y))}{\theta}\nnum\\&=&\lim_{\theta\rightarrow 0^{+}}\frac{c-c}{\theta}=0,\quad y\in \Omega(g^{0}),\nnum\\\EN
since $\varphi(\theta,y)\in \Omega(g^{0})$ by (\ref{15}) and $v(\Omega(g^{0}))=c$ by (\ref{4}). 

It remains to prove the statement for the sequence of the switching states $\{g\}^{\infty}_{0}=\{g^{0},g^{1},\cdots\}$. 
The switching sequence $\{g\}^{\infty}_{0}$ consists of the switching points on $\varphi(\theta,g^{0})$ which by  (\ref{13}) is an infinite sequence. 

The precompactness of $\Phi^{+}$ with respect to $\mathcal{N}_{b}$ implies the existence of a convergent subsequence of $\{g\}^{\infty}_{0}$ such that
\EQ \label{kir11}\lim_{i\rightarrow \infty}\varphi(\theta^{n}_{i},g^{0})= g^{*}\in \Omega(g^{0}),\Omega(g^{0})\subset  \overline{\Phi^{+}}\subset\overline{\mathcal{N}_{b-\epsilon}}.\nnum\\\EN
Since $v\in C^{\infty}(\mathcal{N}_{b})$ 
\EQ \hspace{1.5cm}\label{kir22}\lim_{i\rightarrow \infty}\nabla v(\varphi(\theta^{n}_{i},g^{0}))=\nabla v(g^{*}),\EN
and
\EQ\label{kos} \lim_{i\rightarrow \infty}\frac{dv(\varphi(\theta^{n}_{i},g^{0}))}{d\theta}|_{\theta=0}=\frac{dv(g^{*})}{d\theta}|_{\theta=0}.\EN
But since the state $\varphi(\theta^{n}_{i},g^{0}))$ is a switching state chosen from the switching sequence $\{g\}^{\infty}_{0}$,  
\EQ \label{kos2}&&\frac{dv(\varphi(\theta^{n}_{i},g^{0}))}{d\theta}|_{\theta=0}=\nnum\\&&-\textbf{I}(TL_{\varphi(\theta^{n}_{i},g^{0})^{-1}}\nabla v,TL_{\varphi(\theta^{n}_{i},g^{0})^{-1}}\nabla v),\EN
As is stated in (\ref{kir11}), the limit point $g^{*}$ is an element of the limit set $\Omega(g^{0})$, therefore by (\ref{yy}) we have 
\EQ\hspace{3cm}\frac{dv(g^{*})}{d\theta}|_{\theta=0}=0. \EN

From (\ref{kir22})-(\ref{kos2}) we have
\EQ 0&=&\frac{dv(x^{*})}{d\theta}|_{\theta=0}=\lim_{i\rightarrow \infty}\frac{dv(\varphi(\theta^{n}_{i},x^{0}))}{d\theta}|_{\theta=0}\nnum\\&=&\lim_{i\rightarrow \infty}\big(-\textbf{I}(TL_{\varphi(\theta^{n}_{i},g^{0})^{-1}}\nabla v,TL_{\varphi(\theta^{n}_{i},g^{0})^{-1}}\nabla v)\big).\nnum\\\EN
Hence
\EQ \nabla v(g^{*})=0,\EN
or equivalently
\EQ dv|_{g^{*}}=0. \EN
But by \textit{\textbf{H1}}$, g_{*}$ is the unique point in $\mathcal{N}_{b-\epsilon}\subset \mathcal{N}_{b}$ for which this holds, hence all subsequences of $\{g\}^{\infty}_{0}$ converge to $g_{*}=g^{*}$ an hence so does the sequence.
\end{proof}

\begin{definition}\label{d5} (\textit{EG-HMP Algorithm})\\
 Consider the hybrid system (\ref{1}) with two distinct discrete states and the performance function $v(.)$. 
 \begin{itemize}
 \item Set $k=0$ and initialize the algorithm with $g^{k}=g(0)\in G$.
 
\item  For a given $0<\beta$, compute $\nabla v(g^{k})$. If $\textbf{I}(TL_{g^{k^{-1}}}\nabla v,TL_{g^{k^{-1}}}\nabla v)< \beta$,  then stop, else 
\EQ&&\gamma_{g^{k}}(\theta)=g(\theta)=\hspace{0cm}g^{k}\star \exp(-\theta TL_{g^{k^{-1}}}\nabla v(g^{k})), \nnum\\&& \theta\in[\theta^{k},\theta^{k+1}),\quad g(\theta)\in G,\nnum\EN
  where
  \EQ &&g^{k+1}=g(\theta^{k+1}),\hspace{.2cm}\nnum\\&&\theta^{k+1}=\sup\limits_{s}\{s; \frac{dv(g(t))}{dt}\leq 0, t\in [\theta^{k},s), \nnum\\&& s\in[\theta^{k},\theta^{k}+1)\}.\nnum\EN

\item     Set $k:=k+1$ and go to step two.
\end{itemize}\halmos
\end{definition}

\begin{theorem}
\label{t3}
Assume \textit{\textbf{H1}} holds for $\mathcal{N}_{b}\subset G$ and $g^{0}\in \mathcal{N}_{b}$, for the HOCP with the performance function $v(.)$. Then there exists  a single finite k at which the algorithm stops and 
either:\\

(i): $0<\textbf{I}(TL_{g^{k^{-1}}}\nabla v,TL_{g^{k^{-1}}}\nabla v)< \beta$,\\

or\\

(ii): $\textbf{I}(TL_{g^{k^{-1}}}\nabla v,TL_{g^{k^{-1}}}\nabla v)=0,$ in which case the Geodesic-Gradient flow, $\varphi$,
reaches an equilibrium after a finite number of switchings and hence $\nabla v(g^{k(\beta)}(\beta))=0$ and  $g^{k(\beta)}(\beta)=g_{*}$, where $g_{*}$ is the unique point of $\mathcal{N}_{b}\subset G$ such that $||\nabla v(g_{*})||=0$.\\

 In either case, $g^{k(\beta)}(\beta)$ is such that 
\EQ \hspace{.6cm}g^{k(\beta)}(\beta)\rightarrow g_{*},\quad  k(\beta)\rightarrow\infty,\quad as \quad  \beta\rightarrow 0.\EN
\end{theorem}
\begin{proof}
The first statement is immediate by the Definition \ref{d5}.
The second holds since $v(.)$ has a unique local minimum at $g_{*}$, and $v(.)\in C^{1}(\mathcal{N}_{b})$ with $\nabla v(g_{*})=0$, 
\EQ &&\rho_{\beta}(g_{*}):=\nnum\\&&\hspace{0cm}\sup\big\{d_{G}(g,g_{*});\textbf{I}(TL_{g^{-1}}\nabla v,TL_{g^{-1}}\nabla v)< \beta, g\in G\big\},\nnum\\\EN
where $d_{G}(.,.)$ is the distance on $G$. Moreover $\rho$ is such that $\rho_{\beta}(g_{*})\rightarrow 0$ as $\beta\rightarrow 0$, hence $g^{k(\beta)}(\beta)\rightarrow g_{*},\quad as \quad \beta\rightarrow 0$, see \cite{pet}.
\end{proof}

\section{Satellite Example}
In this section we give a conceptual example  for a satellite orientation control which is modeled by elements of $SO(3)$. The control inputs in this model are given by the angular velocities in order to display the notion of left invariant hybrid systems optimal control. 

We recall that $SO(3)$ is the rotation group in $\mathds{R}^{3}$ given by 
\EQ SO(3)= \big \{g\in GL(3)| \quad g.g^{T}=I,\hspace{.1cm} det(g)=1\big\},\EN 
where $GL(n)$ is the set of nonsingular $n\times n$ matrices.
The Lie algebra of $SO(3)$ which is denoted by $so(3)$ is given by (see \cite{Varad})
\EQ so(3)=\big\{     X\in M(3)|\quad X+X^{T}=0\big\},\EN
where $M(n)$ is the space of all $n\times n$ matrices. The Lie group operation $\star$ is given by the matrix multiplication and consequently $TL_{g_{2}}$ is also given by the matrix multiplication $g_{2}X,\hspace{.2cm} X\in T_{g_{1}}G$.

A left invariant dynamical system on $SO(3)$ is given by 
\EQ \dot{g}(t)=gX,\quad g(0)=g_{0},\hspace{.2cm} X\in so(3).\EN
The Lie algebra bilinear operator is defined as the commuter of matrices, i.e. 
\EQ [X,Y]=XY-YX,\quad X,Y\in so(3).\EN 
The kinematic equations expressing the state trajectory $g(.)$ for a satellite is given by
\EQ\label{po} &&\left\{\begin{array}{ll}\dot{g}(t)=g(t)X(t)\\\dot{\hat{X}}(t)+\mathds{I}^{-1}(\hat{X}(t)\times\mathds{I}\hat{X}(t))=\mathds{I}^{-1}\tau(t)\end{array},\right.\nnum\\&& g(t)\in SO(3), X(t)\in so(3).\EN
The first line of the equation above is written as
\EQ &&\left( \begin{array}{ll} \dot{g_{11}} \quad \dot{g_{12}} \quad \dot{g_{13}}\\\dot{g_{21}}\quad \dot{g_{22}}\quad \dot{g_{23}}\\\dot{g_{31}}\quad\dot{g_{32}}\quad \dot{g_{33}}
        \end{array}\right)=\nnum\\&&\left( \begin{array}{ll} g_{11} \quad g_{12} \quad g_{13}\\g_{21}\quad g_{22}\quad g_{23}\\g_{31}\quad g_{32}\quad g_{33}
        \end{array}\right)\left( \begin{array}{ll} 0 \quad \hspace{.5cm}X_{1}(t) \quad X_{3}(t)\\-X_{1}(t)\quad 0\quad\hspace{.3cm} X_{2}(t)\\-X_{3}(t)\quad \hspace{-.2cm}-X_{2}(t)\quad 0
        \end{array}\right),\nnum\\\EN
and  $\hat{.}:so(3)\rightarrow \mathds{R}^{3}$ is an isomorphism such that

\EQ \widehat{\left( \begin{array}{ll} 0 \quad \hspace{.5cm}X_{1}(t) \quad X_{3}(t)\\-X_{1}(t)\quad 0\quad\hspace{.3cm} X_{2}(t)\\-X_{3}(t)\quad \hspace{-.2cm}-X_{2}(t)\quad 0
        \end{array}\right)}=(X_{1}(t),X_{2}(t),X_{3}(t)).\EN
 $\mathds{I}$ is the inertia tensor given by $\left[ \begin{array}{ll} \mathds{I}_{1} \quad 0 \quad 0\\0\quad \mathds{I}_{2}\quad 0\\0\quad 0\quad \mathds{I}_{3}
        \end{array}\right]$, the input torque is $t\rightarrow\tau(t)\in \mathds{R}^{3}$ and $\times $ is the cross product in $\mathds{R}^{3}$. For more details of the modeling above see \cite{Lewis}, page 281. The second part of (\ref{po}) is the \textit{controlled Euler-Poincare equation} and (\ref{po}) is the geodesic equation on $G$ in the presence of external forces, see \cite{Lewis, Arnold}.

We simplify the satellite model above in order to be able to give an explicit
computational example
of Theorem \ref{t22}. In this example we consider $X\in so(3)$ as the control, i.e. $\hat{X}(t)=(u_{1}(t),u_{2}(t),u_{3}(t))\in \mathds{R}^{3}$. A controlled left invariant system on $SO(3)$ is then defined by
\EQ &&\hspace{-.2cm}\begin{array}{ll}\dot{g}(t)=g(t)\left( \begin{array}{ll} 0 \quad \hspace{.5cm}u_{1}(t) \quad u_{3}(t)\\-u_{1}(t)\quad 0\quad\hspace{.3cm} u_{2}(t)\\-u_{3}(t)\quad \hspace{-.2cm}-u_{2}(t)\quad 0
        \end{array}\right),\end{array} \nnum\\&&g(t)\in SO(3), (u_{1},u_{2},u_{3})\in \mathds{R}^{3}.\EN
        The Lie algebra $so(3)$ is spanned by $e_{1}=\left( \begin{array}{ll} 0 \quad\hspace{.3cm} 1 \quad 0\\-1\quad 0\quad 0\\0\quad\hspace{.3cm} 0\quad 0
        \end{array}\right), e_{2}=\left( \begin{array}{ll} 0 \quad\hspace{.3cm} 0 \quad 0\\0\quad\hspace{.3cm} 0\quad 1\\0\quad -1\quad 0
        \end{array}\right) \mbox{and}\hspace{.2cm}\nnum\\ e_{3}=\left( \begin{array}{ll} 0 \quad\hspace{.3cm} 0 \quad 1\\0\quad\hspace{.3cm} 0\quad 0\\-1\quad 0\quad 0
        \end{array}\right)$.
        One can check that
        \EQ \label{kir}[e_{1},e_{2}]=e_{3},\hspace{.1cm}[e_{1},e_{3}]=-e_{2},\hspace{.1cm}[e_{2},e_{3}]=e_{1}.\EN
By the controllability results presented in \cite{Jurd} since all the Lie algebras generated by $(e_{1},e_{2}),(e_{2},e_{3}),(e_{1},e_{3})$ span the tangent space of the Lie group all the systems derived by each pair of controls are controllable. Here we define a controlled hybrid system (no switching manifolds) on $SO(3)$ as follows:
The continuous dynamics are given by
\EQ &&\hspace{-.2cm}\begin{array}{ll}\dot{g}_{1}(t)=g_{1}(t)\left( \begin{array}{ll} 0 \quad \hspace{.5cm}u_{1}(t) \quad 0\\\hspace{-.2cm}-u_{1}(t)\quad 0\quad\hspace{.3cm} u_{2}(t)\\0\quad\hspace{.2cm} -u_{2}(t)\quad 0
        \end{array}\right),t\in[t_{0},t_{s})\\\dot{g}_{2}(t)=g_{2}(t)\left( \begin{array}{ll} 0 \quad \hspace{.5cm}u_{1}(t) \quad u_{3}(t)\\-u_{1}(t)\quad 0\quad\hspace{.3cm} 0\\-u_{3}(t)\quad 0\quad\hspace{.3cm} 0
        \end{array}\right),t\in[t_{s},t_{f}]\\g_{1}(t),g_{2}(t)\in SO(3), (u_{1},u_{2},u_{3})\in \mathds{R}^{3},\end{array}\EN
        where
        \EQ J_{1}=\frac{1}{2}\int^{t_{s}}_{t_{0}}u^{2}_{1}(t)+u^{2}_{2}(t)dt, J_{2}=\frac{1}{2}\int^{t_{s}}_{t_{0}}u^{2}_{1}(t)+u^{2}_{3}(t)dt.\EN
       
       \begin{figure}
\begin{center}
\hspace*{-.5cm}\includegraphics[scale=.45]{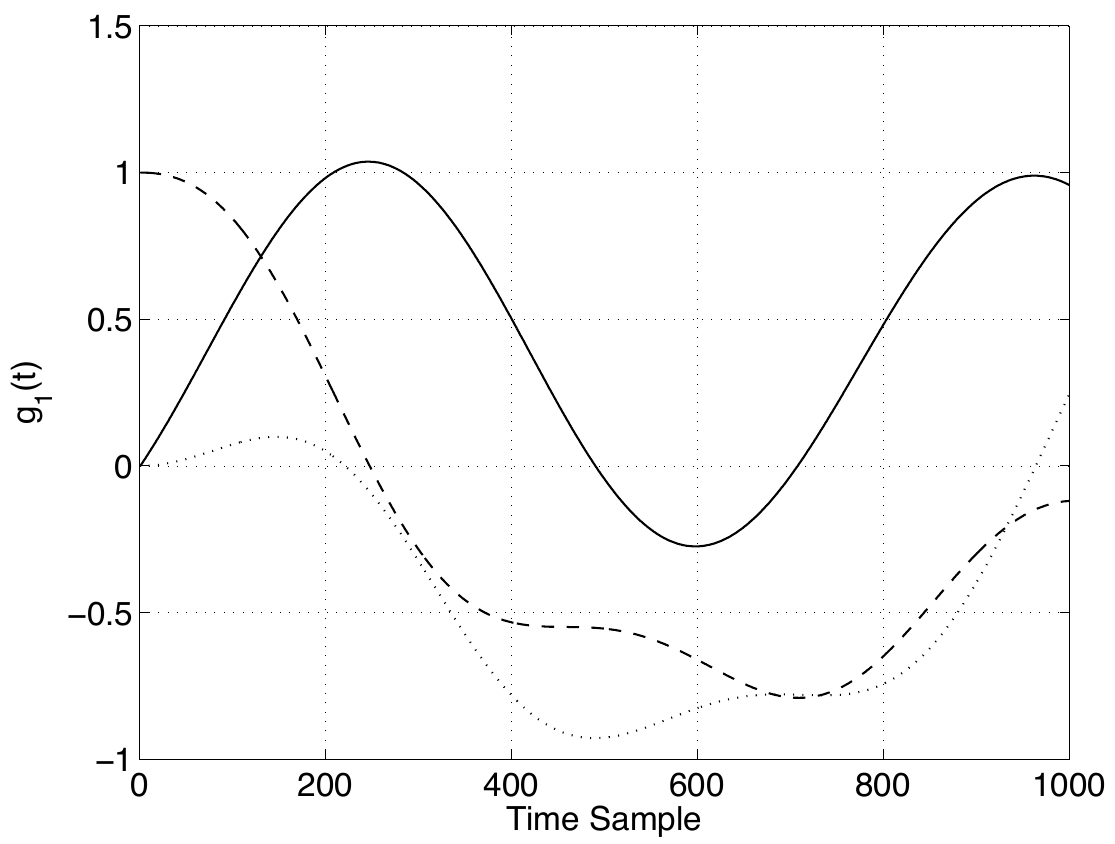}
      \caption{ Hybrid State Trajectory (discrete state  1)}
      \label{f1}
      \end{center}
   \end{figure}

\begin{figure}
\begin{center}
\hspace*{-.5cm}\includegraphics[scale=.45]{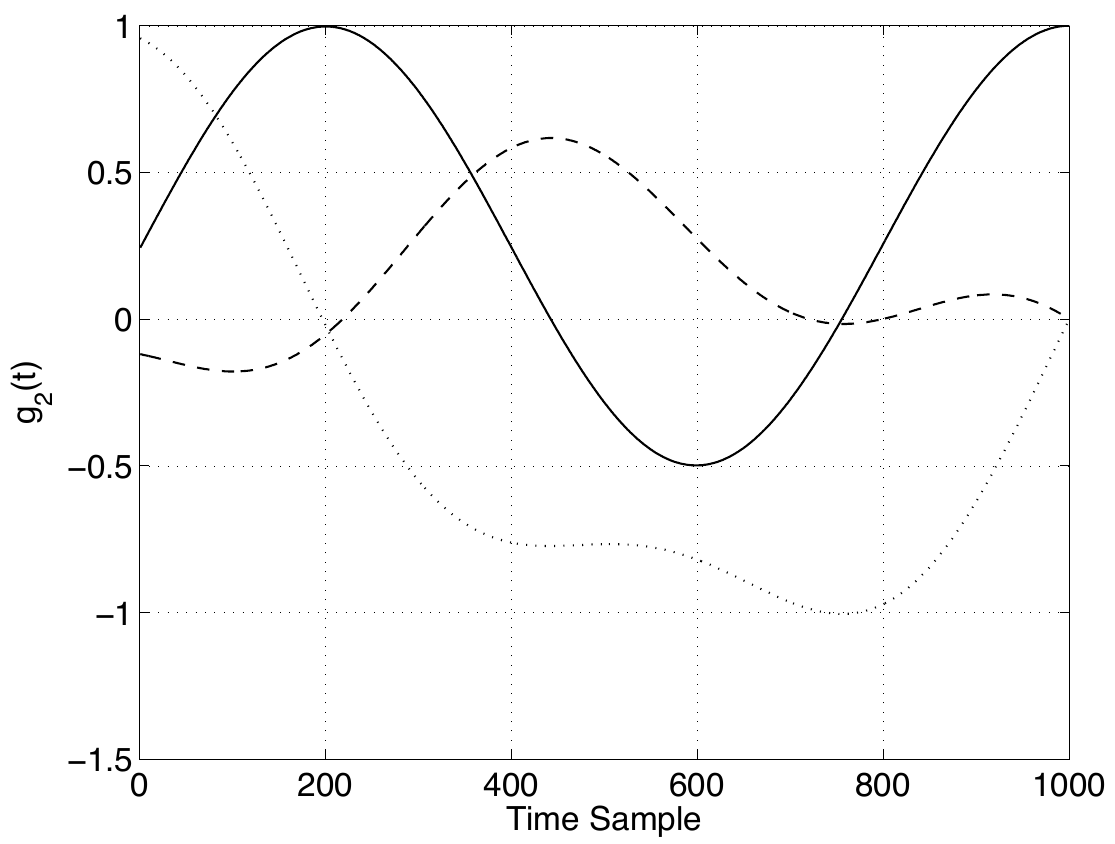}
      \caption{ Hybrid State Trajectory (discrete state 2)}
      \label{f2}
      \end{center}
   \end{figure}
 \begin{figure}
\begin{center}
\hspace*{-1.3cm}\includegraphics[scale=.45]{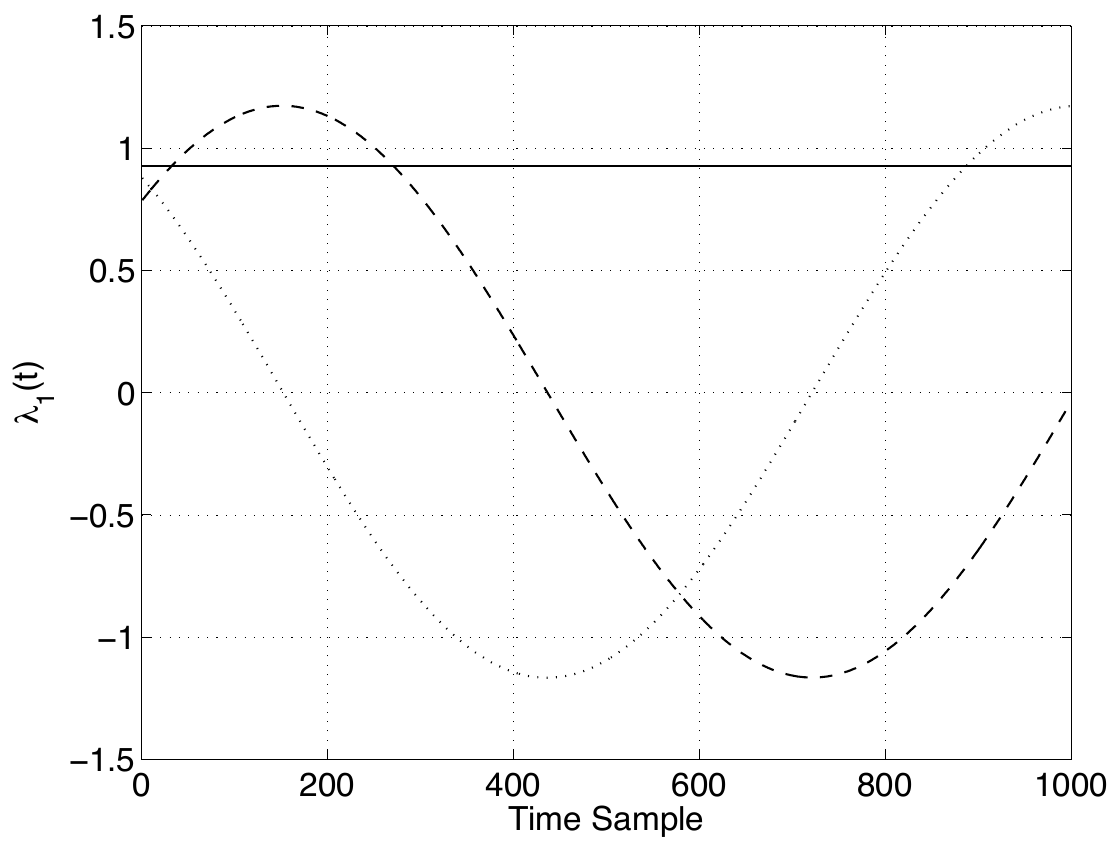}
      \caption{ Hybrid Adjoint Trajectory (discrete state 1)}
      \label{f3}
      \end{center}
   \end{figure}

 \begin{figure}
\begin{center}
\hspace*{-1.3cm}\includegraphics[scale=.45]{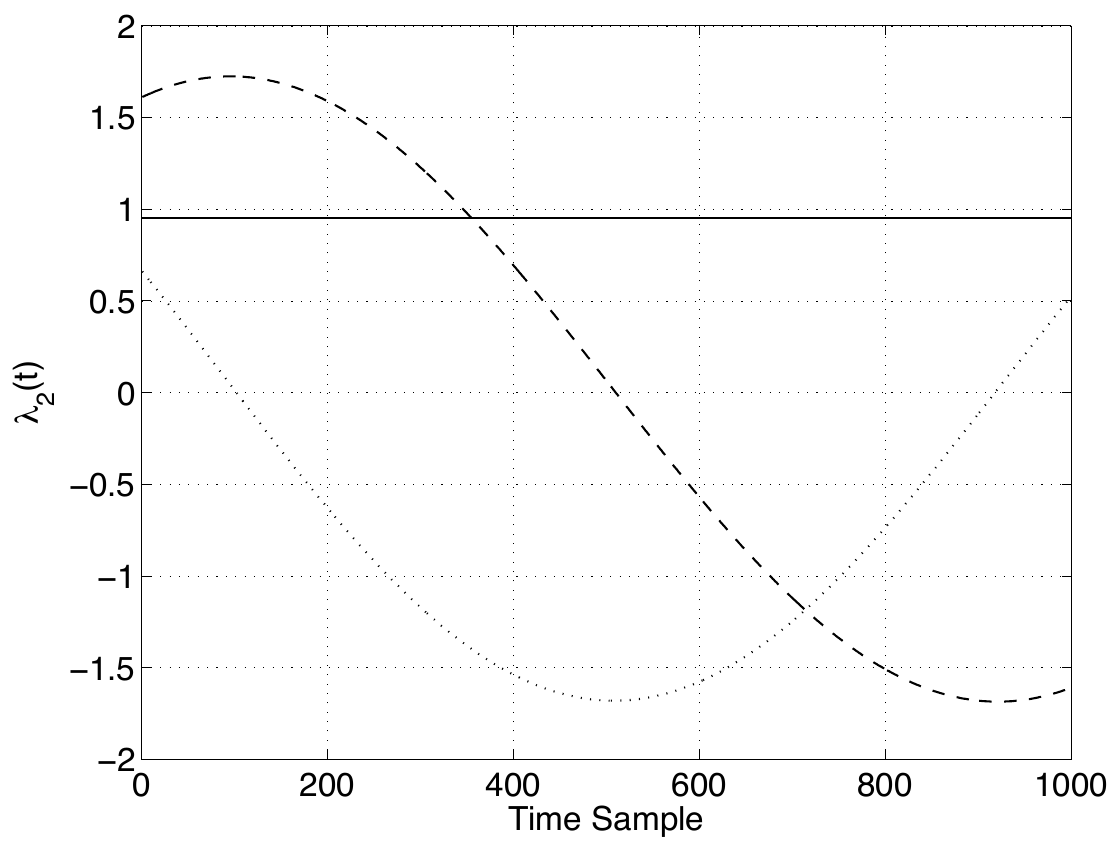}
      \caption{ Hybrid Adjoint Trajectory (discrete state 2)}
      \label{f4}
      \end{center}
   \end{figure}
\begin{figure}
\begin{center}
\hspace*{-1.3cm}\includegraphics[scale=.45]{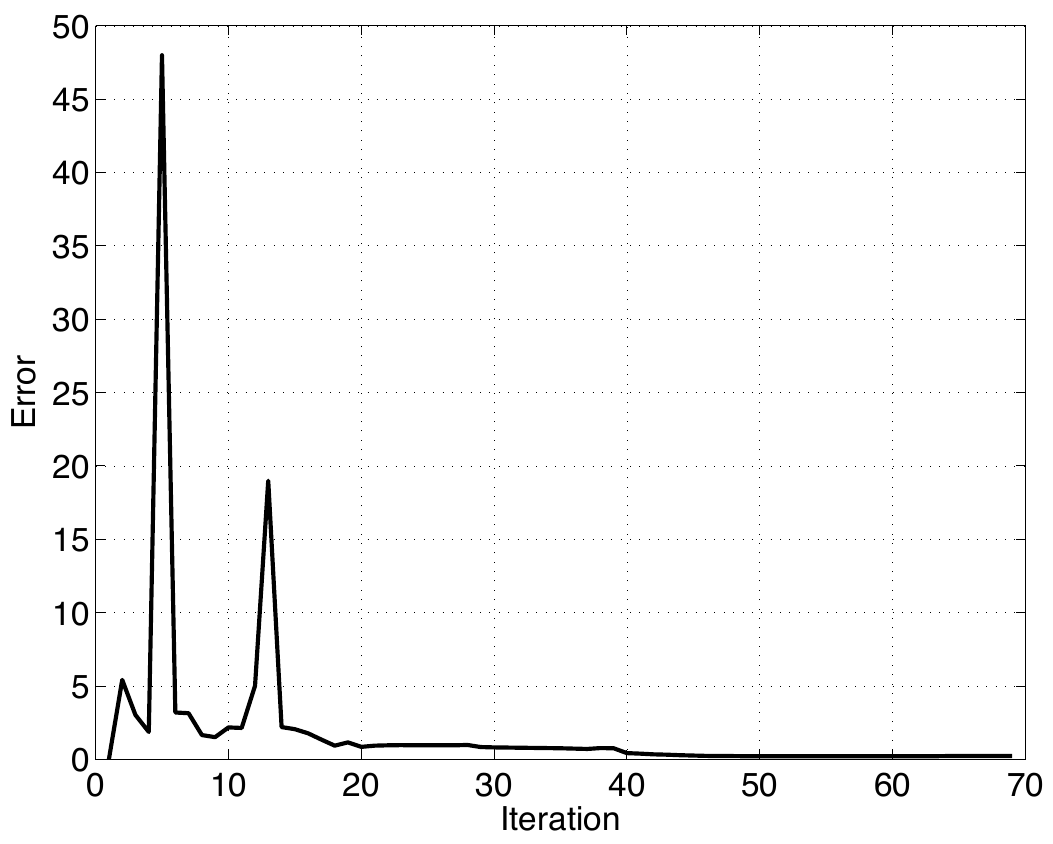}
      \caption{ EX-HMP Convergence}
      \label{f5}
      \end{center}
   \end{figure}
        The Hamiltonians corresponding to the left invariant dynamics are 
        \EQ H_{1}(\lambda,u_{1},u_{2})=\langle \lambda, u_{1}e_{1}+u_{2}e_{2}\rangle+\frac{1}{2}(u^{2}_{1}+u^{2}_{2}),\EN
         \EQ H_{2}(\lambda,u_{1},u_{3})=\langle \lambda, u_{1}e_{1}+u_{3}e_{3}\rangle+\frac{1}{2}(u^{2}_{1}+u^{2}_{3}),\EN
         where $\lambda=\lambda_{1}e^{*}_{1}+\lambda_{2}e^{*}_{2}+\lambda_{3}e^{*}_{3}$ and $\langle e^{*}_{i},e_{j}\rangle=\delta_{i,j},\hspace{.2cm}i,j=1,2,3$. By the Minimum Principle, the optimal controls are obtained as
         \EQ u^{*}_{1}(t)=-\lambda_{1}(t), u^{*}_{2}(t)=-\lambda_{2}(t), t\in[t_{0},t_{s}),  \EN
         \EQ u^{*}_{1}(t)=-\lambda_{1}(t), u^{*}_{3}(t)=-\lambda_{3}(t), t\in[t_{s},t_{f}].  \EN
         By (\ref{kir}) we have
         \EQ\label{a1} ad_{e_{1}}=\left( \begin{array}{ll} 0 \quad 0 \quad 0\\0\quad0\quad -1\\0\quad1\quad  \hspace{.3cm}0
        \end{array}\right),\EN
        \EQ \label{a2}ad_{e_{2}}=\left( \begin{array}{ll} 0 \quad 0 \quad 1\\0\quad0\quad 0\\\hspace{-.25cm}-1\quad0\quad 0
        \end{array}\right),\EN
 \EQ \label{a3}ad_{e_{3}}=\left( \begin{array}{ll} 0 \quad \hspace{-.3cm}-1 \quad 0\\1\quad0\quad 0\\0\quad0\quad 0
        \end{array}\right).\EN
And
\EQ \frac{\partial H_{1}}{\partial \lambda}=u_{1}e_{1}+u_{2}e_{2}, \frac{\partial H_{2}}{\partial \lambda}=u_{1}e_{1}+u_{3}e_{3},\EN
therefore
\EQ ad_{\frac{\partial H_{1}}{\partial \lambda}}=\left( \begin{array}{ll} 0 \quad 0 \quad u_{2}\\0\quad0\quad \hspace{-.3cm}-u_{1}\\\hspace{-.3cm}-u_{2}\quad u_{1}\quad 0
        \end{array}\right),ad_{\frac{\partial H_{2}}{\partial \lambda}}=\left( \begin{array}{ll} 0 \quad \hspace{-.3cm}-u_{3} \quad 0\\u_{3}\quad0\quad \hspace{-.3cm}-u_{1}\\0\quad u_{1}\quad 0
        \end{array}\right).\EN
        Hence the differential equations corresponding to the adjoint variable $\lambda$ are given by
        \EQ&& \begin{array}{ll}\dot{\lambda_{1}}(t)=\lambda_{3}(t)u^{*}_{2}(t),\\\dot{\lambda_{2}}(t)=-\lambda_{3}(t)u^{*}_{1}(t),\\\dot{\lambda_{3}}(t)=-\lambda_{1}(t)u^{*}_{2}(t)+\lambda_{2}(t)u^{*}_{1}(t),\quad t\in [t_{0},t_{s}),\end{array}\EN
\EQ&&\hspace*{-2.25cm}\begin{array}{ll} \dot{\lambda_{1}}(t)=-\lambda_{2}(t)u^{*}_{3}(t),\\\dot{\lambda_{2}}(t)=\lambda_{1}(t)u^{*}_{3}(t)-\lambda_{3}(t)u^{*}_{1}(t),\\\dot{\lambda_{3}}(t)=\lambda_{2}(t)u^{*}_{1}(t),\quad \hspace{2cm}t\in [t_{s},t_{f}].\end{array}\EN

 \begin{definition}\label{d7}
 For a finite dimensional Lie algebra $so(3)$, we define the Killing Form  $B(.,.):so(3)\times so(3)\rightarrow \mathds{R}$ as
 \EQ B(X,Y)=tr(ad_{X}ad_{Y}),\hspace{.2cm}X,Y\in so(3).\EN 
 \end{definition}
  The Killing Form is invariant in the sense that
   \EQ B([X,Y],Z)=B(X,[Y,Z]).\EN
Now corresponding to $B$ we introduce an inner product $I_{B}$ on $so(3)$ such that
\EQ I_{B}(X,Y)=-tr(ad_{X}ad_{Y}).\EN 
 Lemma \ref{l10} implies that $B$ induces a left invariant metric on $G$. By (\ref{a1})-(\ref{a3}) we have
 \EQ I_{B}=\left( \begin{array}{ll} 2 \quad 0 \quad 0\\0\quad2\quad 0\\0\quad0\quad  2
        \end{array}\right)\EN
        By the realization above, $T^{*}L_{g}dv=\lambda_{1}e^{*}_{1}+\lambda_{2}e^{*}_{2}+\lambda_{3}e^{*}_{3}\in so^{*}(3)$ implies that
        \EQ TL_{g^{-1}}\nabla v=\frac{\lambda_{1}}{2}e_{1}+\frac{\lambda_{2}}{2}e_{2}+\frac{\lambda_{3}}{2}e_{3}\in so(3).\EN 
         The algorithm initiates from $t_{0}=0$, $t_{f}=10$, $g_{0}=\left( \begin{array}{ll} 0 \quad\hspace{.2cm} 0 \quad \hspace{.1cm}1\\\hspace{-.1cm}0\quad\hspace{.1cm} -1\quad 0\\1\quad \hspace{.2cm}0\quad \hspace{.2cm}0
        \end{array}\right)$, $t_{s}=5.8 s$ and $g_{s}=\left( \begin{array}{ll} 0 \quad\hspace{.2cm} 1 \quad \hspace{.1cm}0\\\hspace{-.1cm}-1\quad\hspace{.1cm} 0\quad 0\\0\quad \hspace{.2cm}0\quad 1
        \end{array}\right)$ and $g_{f}=\left( \begin{array}{ll} 1 \quad\hspace{.2cm} 0 \quad \hspace{.1cm}0\\\hspace{0cm}0\quad\hspace{.1cm} 1\quad 0\\0\quad \hspace{.2cm}0\quad 1
        \end{array}\right)$. 
        The algorithm converges to $g_{s}=\left( \begin{array}{ll} 0.3039 \quad\hspace{.2cm}  0.9574 \quad \hspace{.1cm}-0.1194\\\hspace{0cm}-0.3688\quad\hspace{.1cm} 0.1508\quad  -0.9165\\-0.8604\quad \hspace{.2cm}0.3156\quad\hspace{.2cm} 0.3988
        \end{array}\right)$ and $t_{s}=5.9733$. The state trajectory and adjoint variables are shown in Figures \ref{f1}-\ref{f4} and Figure \ref{f5} shows the convergence of the Exp-HMP algorithm.
 

 \bibliographystyle{ieeetr}
\bibliography{paper}

\end{document}